\documentclass{amsart}
\usepackage{graphicx, enumerate, url}
\usepackage[margin=1in]{geometry}
\usepackage{amssymb,mathtools}
\usepackage{color}
\usepackage{kbordermatrix}
\numberwithin{equation}{section}

\theoremstyle{plain}
\newtheorem{theorem}{Theorem}[section]
\newtheorem{proposition}[theorem]{Proposition}
\newtheorem{lemma}[theorem]{Lemma}
\newtheorem{corollary}[theorem]{Corollary}
\newtheorem{conjecture}[theorem]{Conjecture}
\newtheorem{algorithm}[theorem]{Algorithm}

\theoremstyle{definition}
\newtheorem{definition}[theorem]{Definition}
\newtheorem{example}[theorem]{Example}

\newtheorem{question}[theorem]{Question}

\theoremstyle{remark}

\let\H\undefined
\let\S\undefined

\DeclareMathOperator{\rank}{rank}
\DeclareMathOperator{\sgn}{sgn}

\DeclareMathOperator{\Id}{Id}

\DeclareMathOperator{\brank}{brank}
\DeclareMathOperator{\Hank}{H}
\DeclareMathOperator{\H}{H}
\DeclareMathOperator{\S}{S}
\DeclareMathOperator{\T}{T}

\newcommand{\mA}{\mathcal{A}}

\newcommand{\mH}{\mathcal{H}}
\newcommand{\mc}{\mathcal}
\newcommand{\re}{\mathbb{R}}
\newcommand{\N}{\mathbb{N}}

\newcommand{\cpx}{\mathbb{C}}
\newcommand{\bbm}{\begin{bmatrix}}
\newcommand{\ebm}{\end{bmatrix}}
\newcommand{\lmd}{\lambda}
\newcommand{\bca}{\begin{cases}}
\newcommand{\eca}{\end{cases}}
\newcommand{\bit}{\begin{itemize}}
\newcommand{\eit}{\end{itemize}}

\newcommand{\reff}[1]{(\ref{#1})}
\newcommand{\bea}{\begin{eqnarray}}
\newcommand{\eea}{\end{eqnarray}}
\newcommand{\be}{\begin{equation}}
\newcommand{\ee}{\end{equation}}

\begin{document}
\title{Hankel Tensor Decompositions and Ranks}

\author[Jiawang Nie]{Jiawang~Nie}
\address{Department of Mathematics, University of California San Diego,
9500 Gilman Drive
La Jolla, CA 92093}
\email{njw@math.ucsd.edu}

\author[Ke Ye]{Ke~Ye}
\address{
Academy of Mathematics and System Sciences,
Chinese Academy of Science, Beijing 100190, China
}
\email{keyk@amss.ac.cn}

\begin{abstract}
Hankel tensors are generalizations of Hankel matrices.
This article studies both the computational and
algebraic aspects of Hankel tensor ranks. We prove that for a low rank symmetric tensor,
there exists a base change to make it a Hankel tensor. We also provide an algorithm
that can compute the Vandermonde ranks and decompositions
for all Hankel tensors. For a generic $n$-dimensional
Hankel tensor of even order or order three,
we prove that the candecomp-parafac rank, symmetric rank, Vandermonde rank, border rank, symmetric border rank, and Vandermonde border rank all
coincide with each other. Some open questions are also posed.
\end{abstract}

\keywords{Hankel tensor, rank, Vandermonde decomposition, Waring decomposition}

\subjclass[2010]{15A18, 15A69}

\maketitle

\section{Introduction}\label{sec: 1}

\subsection{Various ranks for tensors}
\label{subsec:ranks}

For integers $m,n>0$, denote by $\operatorname{T}^m(\cpx^n)$ the space
of all $n$-dimensional complex tensors of order $m$.
A tensor $\mA \in \operatorname{T}^m(\cpx^n)$ is an array
indexed by an integer tuple $(i_1, \ldots, i_m)$ in the range
$1 \leq i_1, \ldots, i_m \leq n $, that is,
\[
\mA = ( \mA_{i_1 \ldots i_m})_{ 1 \leq i_1, \ldots, i_m \leq n}.
\]
The tensor $\mA$ is {\it symmetric} if
$\mA_{i_1 \ldots i_m}$ is invariant with respect to all permutations
of $(i_1,\ldots, i_m)$. Denote by $\operatorname{S}^m(\cpx^n)$ the vector space
of all $n$-dimensional complex symmetric tensors of order $m$.

There are various types of ranks for tensors,
which measure the complexity of tensor computations from different aspects.
The typical ones are the candecomp-parafac (cp) rank \cite{Hitchcock1927}
(also called just simply the rank, or the canonical polyadic rank, in some references),
multilinear rank \cite{LMV2000}, tensor network rank \cite{YL2017},
and nuclear rank \cite{FL2017,LC2010}
(the nuclear rank of a tensor is defined to be the smallest length of
its decompositions that achieve the nuclear norm).
The symmetric rank \cite{CGLM2008} is defined for symmetric tensors.
The Vandermonde rank \cite{Qi:Han} is defined for Hankel tensors.
The border rank for general tensors and
symmetric border rank for symmetric tensors are also defined
for studying algebraic properties.
We refer to \cite{Landsberg2012,Lim13} for various definitions of tensor ranks.
For convenience of reading, we shortly review them in the sequel.

All tensors can be expressed as linear combinations of outer
products of vectors. For $u_1, \ldots, u_m \in \cpx^n$,
the outer product $u_1 \otimes \cdots \otimes u_m$
is the tensor in $\operatorname{T}^m(\cpx^n)$ such that
for all $1 \leq i_1,\ldots,i_m \leq n$
\[
(u_1 \otimes \cdots \otimes u_m)_{i_1 \ldots i_m} = (u_1)_{i_1}
\cdots (u_m)_{i_m}.
\]
The tensors of the form $u_1 \otimes \cdots \otimes u_m$ are called rank-$1$ tensors.
The \emph{cp rank} of $\mA \in \operatorname{T}^m(\cpx^n)$ is defined as
\be \label{df:trank}
\rank(\mA) \coloneqq \min \left\{ r\in \mathbb{N}:
\mA =\sum_{i=1}^r u_{i,1} \otimes \cdots \otimes u_{i, m},
\, u_{i,j} \in \cpx^n
 \right\}.
\ee
%
%
If $\rank (\mathcal{A}) =r$,
the corresponding decomposition in \reff{df:trank}
is called a \emph{rank decomposition} of $\mathcal{A}$.
The \emph{border rank} of $\mathcal{A}$ is then defined as
\[
\brank (\mathcal{A}) \coloneqq \min \left\{  r\in \mathbb{N}:
\mathcal{A} = \lim_{p\to \infty}\sum_{i=1}^r u_{i,1}^{(p)}\otimes \cdots \otimes u_{i,m}^{(p)}, \, \,   u_{i,j}^{(p)} \in \mathbb{C}^n \right\}.
\]
Clearly, it always holds that
$
\brank(\mathcal{A}) \le \rank(\mathcal{A}).
$

For symmetric tensors, we are typically interested in their symmetric ranks.
For $\mA \in \operatorname{S}^m(\cpx^n)$,
its symmetric rank is defined as
\be \label{def:Srank}
\rank_S(\mA) \coloneqq \min \left\{ r\in \mathbb{N}:
\mA = \sum_{i=1}^r (u_i)^{\otimes m},
\,\, u_{i} \in \cpx^n
 \right\}.
\ee
In the above,
$ (u_i)^{\otimes m} := u_i \otimes \cdots \otimes u_i$,
where $u_i$ is repeated $m$ times.
If $\rank_S(\mathcal{A}) = r$,
the corresponding decomposition in \reff{def:Srank}
is called a \emph{symmetric rank decomposition} of $\mathcal{A}$.
A symmetric tensor $\mA \in \operatorname{S}^m(\cpx^n)$
can be uniquely represented by a homogeneous polynomial (i.e., a form)
of degree $m$ and in $(x_1,\ldots, x_n)$, which is
\[
\mA(x) := \sum_{1 \leq i_1, \ldots, i_m \leq n}
\mA_{i_1 \cdots i_m} x_{i_1} \cdots x_{i_m}.
\]
A symmetric decomposition of $\mathcal{A}$ is equivalent to
the decomposition of $\mA(x)$ as a sum of powers of linear forms.
In the literature, the symmetric rank decomposition
is also called a \emph{Waring decomposition},
and the symmetric rank is also called \emph{Waring rank}
\cite{Landsberg2012,OedOtt13}.
The symmetric rank of a form means the symmetric rank of
the corresponding symmetric tensor.
The \emph{symmetric border rank} of
$\mathcal{A} \in \operatorname{S}^m(\mathbb{C}^n)$ is then defined as
\[
\brank_S (\mathcal{A}) \, := \, \min \left\{  r:
\lim_{p\to \infty}\sum_{i=1}^r {u_{i}^{(p)}}^{\otimes m}
= \mathcal{A}, \, u_{i}^{(p)} \in \mathbb{C}^n \right\}.
\]
For a symmetric tensor $\mA$, it is straightforward to see that
\[
\brank (\mathcal{A}) \le
\brank_S(\mathcal{A}) \le
\rank_S(\mathcal{A}).
\]

Determining ranks and decompositions of tensors are fundamental questions
in many applications, such as signal processing \cite{C2002,LC2010},
multiway factor analysis \cite{LMV2000},
and computational complexity \cite{BCS1997,YL2016}.
A general survey about applications can be found in \cite{KolBad09}.
It is NP-hard to compute the most ranks\footnote{
There exist some ranks and decompositions which are easy to compute,
for instance, multi-linear rank and higher order singular value decomposition (HOSVD).}
and decompositions of tensors \cite{Hastad1990,HL2013}.
Even for symmetric tensors, the question of computing
their symmetric ranks and Waring decompositions
still remains NP-hard \cite{Shitov16}.
We refer to the work \cite{BBCM13,DeLaMV,DDeLa14,LRA93}
for general tensor decompositions
and refer to the work \cite{BerGimIda11,BCMT10,GPSTD,OedOtt13}
for symmetric tensor decompositions.
Other interesting questions about tensors
include low rank approximations
\cite{SL2008,GKT2013,IAHL2011,NL2014,nie2017low},
uniqueness of tensor decompositions \cite{ChOtVan15,GalMel,Kruskal1977},
symmetric rank of monomials \cite{LT2010,Oeding2016},
defining ideals of low rank tensors \cite{LM2004,LW2007},
and tensor eigenvalues
\cite{CDN,L05sin,NZ15,Q05,Q07}.

%
%
%

%
%

\subsection{Hankel tensors}

A tensor $\mH \in \operatorname{T}^m(\cpx^n)$ is called {\it Hankel} if
$\mH_{i_1 \ldots i_m}$ is invariant whenever the sum
$i_i+\cdots+i_m$ is a constant \cite{PLH2005,Qi:Han}.
In other words, $\mH$
is a Hankel tensor if and only if there exists a vector
$h:=(h_0, h_1, \ldots, h_{(n-1)m})$ such that
\be \label{eqn:defition of h}
\mH_{i_1 \ldots i_m} = h_{i_1 + \cdots + i_m - m}.
\ee
Clearly, each Hankel tensor is also symmetric.
We denote by $\Hank^m(\mathbb{C}^n)$ the linear subspace
of all Hankel tensors in $\S^m(\mathbb{C}^n)$.
It is easy to obtain that the dimension of $\Hank^m(\mathbb{C}^n)$ is $(n-1)m+1$.
As shown by Qi \cite{Qi:Han}, $\mH$ is a Hankel tensor if and only if
it has a {\it Vandermonde decomposition},
i.e., for some $\lmd_i, t_i \in \cpx$,
\[
\mH = \sum_{i=1}^r \lmd_i (1, t_i, \ldots, t_i^{n-1} )^{\otimes m}.
\]
In this paper, we consider the homogenization of the above:
\be \label{Vdcmp:H:aibi}
\mH = \sum_{i=1}^r  (a_i^{n-1} , a_i^{n-2} b_i, \ldots,
a_ib_i^{n-2}, b_i^{n-1} )^{\otimes m}.
\ee

\begin{definition}
For a Hankel tensor $\mH$, the smallest integer
$r$ for \eqref{Vdcmp:H:aibi} to hold is called the {\it Vandermonde rank}
(or just {\it $V$-rank}) of $\mH$, which is denoted by $\rank_V(\mH)$.
The corresponding decomposition is called
the {\it Vandermonde rank decomposition} or just the {\it $V$-rank decomposition}.
The Vandermonde border rank (or just the border $V$-rank) of $\mH$,
which we denote as $\brank_V(\mH)$, is the smallest $r$ such that
$\mH$ is the limit of a sequence of Hankel tensors whose $V$-rank is $r$.
\end{definition}

We remark that the simplest $n$-dimensional Hankel tensor of order $m$ is of the form
\[
(a^{n-1},a^{n-2}b,\dots, ab^{n-2},b^{n-1})^{\otimes m},\quad a,b\in \mathbb{C}
\]
and such a tensor has the Vandermonde rank one. If we consider $\mH \in \H^m(\mathbb{C}^n)$
as a tensor in $\T^m(\mathbb{C}^n)$, its cp rank
$\rank(\mH)$ is defined in \reff{df:trank};
if we consider it as a tensor in $\S^m(\mathbb{C}^n)$,
its symmetric rank $\rank_S(\mH)$ is defined in \reff{def:Srank}.
The border rank $\brank(\mH)$ and symmetric border
rank $\brank_S(\mH)$ are defined in the same way.
%
%

For a Hankel tensor $\mH$, it clearly holds that
\be \label{rl:b<r<s<v}
\brank(\mH) \leq
\rank(\mH) \leq \rank_S(\mH)  \leq \rank_V(\mH).
\ee
For the relations among various ranks of a Hankel tensor $\mH$,
we have the following two simple facts:
\bit
\item [1)] The $V$-rank of $\mH$ is $1$
if and only if $\rank_S(\mathcal{H}) = 1$.
Clearly, if $\rank_V(\mH) = 1$,
then $\mH \ne 0$ and hence $\rank_S(\mH) \geq 1$,
so they are the same by \reff{rl:b<r<s<v}.
Conversely, if $\rank_S(\mathcal{H}) = 1$, then
$\mathcal{H} = v^{\otimes m}$,
for some $v= (v_1,\dots, v_n)\in \mathbb{C}^n$.
Since $\mH$ is Hankel,
$v_{i_1}\cdots v_{i_m} = v_{j_1} \cdots v_{j_m}$
for all $i_1+\cdots +i_m = j_1 + \cdots + j_m$.
In particular, for $i_1 = \cdots = i_m = i$ and
$j_1 = \cdots = j_{m-2} = i, j_{m-1} = i-1, j_{m} = i+1 $,
we get $v_{i-1} v_{i+1} = v_i^2$ for $i =2,\dots, n-1$.
Therefore, we can parametrize $v$ as
$v = (a^{n-1},a^{n-2}b,\dots, b^{n-1})$,
so $\rank_V(\mH) = 1$.
For this case, all the ranks are the same
by \reff{rl:b<r<s<v}.

\item [2)] For the binary case (i.e., $n=2$),
we also have
$\rank_S(\mathcal{H}) = \rank_V(\mathcal{H})$.
This is because for every $2$-dimensional vector
$v = (a,b)$, the tensor power $v^{\otimes m}$
is itself a Vandermonde decomposition.
When $n=2$, it holds that
\[
\Hank^m(\mathbb{C}^2) =  \operatorname{S}^m(\mathbb{C}^2)
\simeq \mathbb{C}^{m+1}.
\]

\eit

Hankel tensors have broad applications.
They were originally defined in signal processing~\cite{PLH2005}
for studying the Harmonic Retrieval problem~\cite{LS2002}.
Moreover, Hankel tensors can also be used to
solve the interpolation problem \cite{TBM2013}.
Recently, Qi~\cite{Qi:Han} studied Vandermonde decompositions
and complete/strong Hankel tensors.
The inheritance properties and sum-of-squares
decompositions for Hankel tensors are studied
by Ding, Qi and Wei~\cite{DQW17}.
Extremal eigenvalues of Hankel tensors
are discussed in Chen, Qi and Wang~\cite{CQW16}.
Some further results about Hankel tensors appear in Chen, Li and Qi~\cite{CLQ16}.

\subsection{Contributions}

The $V$-rank decompositions of Hankel tensors
are closely related to symmetric rank decompositions of binary forms.
Let $d := (n-1)m$ and let $h$ be as in \reff{eqn:defition of h}.
The vector $h$ can be uniquely identified as a binary form of degree $d$:
\be \label{def:h(x,y)}
h(x,y) := {\sum}_{j=0}^d \binom{d}{j} h_j x^j y^{d-j}.
\ee
It can also be thought of as a symmetric binary tensor of order $d$.
By writing $\rank_S(h)$ (resp., $\brank_S(h)$),
we mean the symmetric rank (resp., symmetric border rank)
when $h$ is regarded as the symmetric tensor
represented by the binary form $h(x,y)$. Note that $\rank_S(h)$
is just the Waring rank of the form $h(x,y)$.
The Vandermonde decomposition \reff{Vdcmp:H:aibi} is equivalent to
\begin{equation}\label{eqn: decomposition of h}
h(x,y) = {\sum}_{i=1}^r (a_i x + b_i y)^d.
\end{equation}
The symmetric rank of $h$ is the smallest $r$ in the above.
In Lemma \ref{lemma: V-rank = S-rank}, we will show that
$\rank_V(\mathcal{H}) = \rank_S(h)$.

%
%

\bigskip
This article focuses on various ranks of Hankel tensors.
We mainly address the following two basic questions:
\bit

\item How can we determine the Vandermonde rank and decomposition
of a Hankel tensor?

\item What are the relations among various ranks of a Hankel tensor?

\eit

First, we propose an algorithm (Algorithm~\ref{algorithm:  1})
that can compute the Vandermonde rank and decomposition
for all Hankel tensors. This will be done in Section~\ref{sec:Vdcmp}.

Second, we show that the cp rank, symmetric rank,
border rank, symmetric border rank and Vandermonde rank
are the same for a generic $\mH \in \H^m(\cpx^n)$
when $m$ is even or $m=3$.
In particular, this implies that Comon's conjecture~\cite{Oeding}
is true for generic Hankel tensors of even order or order three.
Moreover, for a specifically given Hankel tensor,
we give concrete conditions for determining these ranks.
This will be done in Sections~\ref{sec:evenorder} and \ref{sec:order=3}.

We give some preliminary results in Section~\ref{sec:preliminary},
and conclude the paper with some open questions and conjectures
in Section~\ref{sec:openqs}.

\section{Preliminaries}
\label{sec:preliminary}

\subsection*{Notation}

The symbol $\N$ (resp., $\re$, $\cpx$) denotes the set of
nonnegative integers (resp., real, complex numbers).
The symbol $\cpx[x] := \cpx[x_1,\ldots,x_n]$
denotes the ring of polynomials in $x:=(x_1,\ldots,x_n)$
over the complex field $\cpx$.
For any $t\in \re$, $\lceil t\rceil$ (resp., $\lfloor t\rfloor$)
denotes the smallest integer not smaller
(resp., the largest integer not bigger) than $t$.
The cardinality of a set $S$ is denoted as $\#S$.
For a matrix $M$, its null space is denoted as $\ker M$.

\subsection{Elementary algebraic geometry}

For basics in algebraic geometry, we refer to \cite{Harris1992}.
A set $X \subseteq \mathbb{C}^n$ is an \emph{algebraic variety}
if there exist polynomials $f_1,\dots, f_s\in \mathbb{C}[x]$ such that
\[
X = \{ a  \in \cpx^n: \,
f_1(a) = \cdots = f_s(a) = 0 \}.
\]
The \emph{Zariski topology} on $\mathbb{C}^n$ is the topology
such that the closed sets are algebraic varieties.
An algebraic variety is called {\it irreducible}
if it is not a union of two proper algebraic varieties.
We need the notion of a generic point in an irreducible algebraic variety $V$.
For a property $P$ on $V$,
we say that a \emph{generic} point in $V$ has the property $P$
if the set of points in $V$ which do not satisfy $P$
is contained in a proper closed subset of $V$
in the Zariski topology.
For instance, a generic point $(x,y) \in \mathbb{C}^n\times \mathbb{C}^n$
uniquely determines a line $L\subset \mathbb{C}^n$ such that $x,y\in L$.
This is because lines passing through $x,y$ are not unique if and only if $x=y$
and the set $\{(x,y)\in \mathbb{C}^n \times \mathbb{C}^n:  x=y \}$ is a proper closed subset
of $\mathbb{C}^n\times \mathbb{C}^n$.
If the property $P$ is clear from the context,
we just say ``a generic point" without mentioning $P$.

%
%

The projective space $\mathbb{P}^n$ consists of all lines in
$\mathbb{C}^{n+1}$, or equivalently,
$\mathbb{P}^n$ is the set of equivalence classes
\[
[(a_0, \dots, a_n)] = \{ (\lambda a_0,\dots, \lambda a_n)
\in \mathbb{C}^{n+1}: \,
(a_0,\dots, a_n)  \ne 0, \, \lambda \ne 0 \}.
\]
A subset $X\subseteq \mathbb{P}^n$ is a \emph{projective variety}
if there exist homogeneous polynomials
$f_1,\dots,f_s\in \mathbb{C}[x_0,\dots, x_n]$ such that
\[
[a]  \in X \, \text{ if and only if } \, f_1(a) = \cdots = f_s(a) = 0.
\]

\subsection{Multilinear algebra}
\label{ssc:mulinalg}

Let $B$ be a vector space of dimension $n$ and let
$\{b_1,\dots, b_n\}$ be a set of basis. For an integer
$1 \le  p\le n$, the \emph{$p$-th exterior power} of $B$,
denoted as $\bigwedge^p B$, is the vector space spanned by
the $\binom{n}{p}$ vectors $b_{j_1}\wedge \cdots \wedge b_{j_p}$
$( 1\le j_1 < \cdots < j_p \le n )$,
where the wedge product $ v_{1}\wedge \cdots \wedge v_{p}$ is defined as
\[
v_1 \wedge \cdots \wedge v_p \coloneqq \sum_{\sigma \in \mathfrak{S}_p}
\sgn (\sigma) v_{\sigma(1)} \otimes \cdots \otimes v_{\sigma(p)}.
\]
In the above, $\mathfrak{S}_p$ is the permutation group on
$p$ elements and $\sgn (\sigma)$ is the sign of the permutation
$\sigma\in \mathfrak{S}_p$. In particular, we have
\[
v \wedge v =0,\quad v\in B.
\]
Clearly, $\bigwedge^p B$
is a linear subspace of $B^{\otimes n}$, with dimension $\binom{n}{p}$.
Moreover,
\[
v_1 \wedge \cdots \wedge v_p = \sgn (\sigma) v_{\sigma(1)}
\wedge \cdots \wedge v_{\sigma(p)},\quad \forall \, \sigma\in \mathfrak{S}_p.
\]
The exterior power $\bigwedge^p B$ is a generalization of
skew-symmetric matrices. Indeed, if $p=2$, then $\bigwedge^2 B$
is simply the vector space of all $n\times n$ skew symmetric matrices.

Let $A,B,C$ be vector spaces of dimensions $m,n,q$ respectively.
Every tensor $T\in A\otimes B \otimes C$ can be regarded as a linear map
$\varphi_T: A^\ast \to B\otimes C$. If we choose bases
$\{a_1,\dots, a_m\}$, $\{b_1,\dots, b_n\}$ and $\{c_1,\dots, c_q\}$
for $A,B,C$ respectively, then we can write $T$ as
\[
T = \sum_{i=1}^m \sum_{j=1}^n \sum_{k=1}^q t_{ijk} a_i\otimes b_j\otimes c_k.
\]
Let $\{\alpha_1,\dots, \alpha_m\}$ be the dual basis of $A^\ast$,
then $\varphi_T$ is given by
\[
\varphi_T(\alpha_i) =  \sum_{j=1}^n \sum_{k=1}^q
t_{ijk} b_j\otimes c_k.
\]
For each integer $1 \le p \le n-1$, define $\varphi_T^{p}$
to be the linear map
\begin{equation}\label{eqn:wedge}
\varphi_T^{p}: \Big( \bigwedge^p B \Big) \otimes A^\ast
\to \Big( \bigwedge^{p+1}B  \Big) \otimes C,
\end{equation}
which is obtained by tensoring $\varphi_T$ with the identity map $\operatorname{Id}_{\bigwedge^p B}: \bigwedge^p B \to \bigwedge^p B$
and projecting the image of $\varphi_T\otimes \operatorname{Id}_{\bigwedge^p B}$
in $\big( \bigwedge^p B \big) \otimes B \otimes C$
onto $\big( \bigwedge^{p+1}B \big) \otimes C$.
Here, the projection of $\bigwedge^p B \otimes B \otimes C$
onto $\bigwedge^{p+1}B \otimes C$ is determined by
\[
(b_1 \wedge \cdots \wedge b_p) \otimes b_{p+1}\otimes c \mapsto (b_1\wedge \cdots \wedge b_{p}\wedge b_{p+1}) \otimes c.
\]
To be more specific, the linear map $\varphi_T^p$ is defined such that
\[
(b_1 \wedge \cdots \wedge b_p)  \otimes \alpha_i \mapsto
\sum_{j=1}^n \sum_{k=1}^q t_{ijk}
(b_1 \wedge \cdots \wedge b_p\wedge b_{j}) \otimes c_k.
\]

\begin{theorem}\cite[Theorem 2.1]{JO2015}
\label{thm:lower bound}
Let $T$ be a tensor in $A\otimes B \otimes C$ and
$0 < p < n$ be an integer. If $\varphi_T^p$
is the linear map defined in \eqref{eqn:wedge}, then
\[
\operatorname{brank}(T)  \ge \operatorname{rank} \varphi_T^p
\Big/ \binom{n-1}{p}.
\]
\end{theorem}

Theorem \ref{thm:lower bound} can be used to obtain Strassen equations
for bounded border rank tensors,
see \cite[p.~81]{Landsberg2012} and \cite[Theorem 3.2]{O2007}.
For instance, when $\dim A = \dim B = \dim C  = 2$ and $p =1$,
a tensor $T \in A \otimes B \otimes C$ can be written as
$T = \sum_{i,j,k=1}^2 t_{ijk} a_i \otimes b_j \otimes c_k$. Hence,
$
\varphi_T(\alpha_i) = \sum_{j,k=1}^{2,2} t_{ijk} b_j \otimes c_k.
$
This implies that $\varphi_T^p : B \otimes A^\ast \to \bigwedge^2 B \otimes C$
is defined as
\[
\varphi_T^p(b \otimes \alpha_i) = \sum_{j,k=1}^{2,2} t_{ijk}
(b_j \wedge b)  \otimes c_k.
\]
It corresponds to the $4 \times 2$ matrix
\[
M_T^p  \coloneqq
\arraycolsep=2.5pt
\kbordermatrix{
& (b_1 \wedge b_2)\otimes c_1 & (b_1 \wedge b_2)\otimes c_2 \\
b_1 \otimes \alpha_1 & -t_{121} & -t_{122}  \\
b_1 \otimes \alpha_2 & -t_{221} & -t_{222}   \\
b_2 \otimes \alpha_1 &  t_{111} & t_{112} \\
b_2 \otimes \alpha_2 & t_{211} &  t_{212} \\
}
\]
The rows of $M_T^p$ are indexed by the basis vectors
$b_1 \otimes \alpha_1, b_1 \otimes \alpha_2, b_2 \otimes \alpha_1, b_2 \otimes \alpha_2$
of $B \otimes A^\ast$ and whose columns are indexed by the basis vectors
$(b_1 \wedge b_2) \otimes c_1, (b_1\wedge b_2) \otimes c_2$ of
$\big( \bigwedge^2 B \big) \otimes C$. For readers' convenience, we label rows and columns with their corresponding bases, respectively. In the literature, the linear map $\varphi_T^p$
defined as in \eqref{eqn:wedge} is called a Young flattening or a Koszul flattening
of the tensor $T$. We refer to
\cite[p.~81]{Landsberg2012} and \cite[Example 4.2]{OedOtt13}.

\subsection{Symmetric ranks of binary forms}

A binary form $h(x,y)$ is a homogeneous polynomial in two variables $x,y$.
Every binary form of degree $d$ can be regarded as a
symmetric binary tensor of order $d$,
so $\operatorname{S}^d(\mathbb{C}^2) \simeq \mathbb{C}^{d+1}$.
The symmetric rank of the symmetric tensor represented by $h(x,y)$
is also called the symmetric rank of $h(x,y)$.
We can write $h(x,y) = \sum_{i=0}^d \binom{d}{i} h_i x^i y^{d-i}$.
For convenience, we denote
\[
h := (h_0, h_1, \ldots, h_d).
\]
For $ 0 \leq r \leq d$, we denote the Hankel matrix
\be \label{eqn:definition of C_{d-r,r}}
C_{d-r, r}(h) := \bbm
h_0  &  h_1  &  \cdots &  h_r \\
h_1  &  h_2  &  \cdots &  h_{r+1}  \\
\vdots & \vdots & \ddots & \vdots \\
h_{d-r} &  h_{d-r+1}  &  \cdots &  h_{d}  \\
\ebm\in \mathbb{C}^{(d-r+1) \times (r+1)}.
\ee

The symmetric rank $\rank_S(h)$ of $h(x,y)$
can be determined as follows.

\begin{theorem} \cite[Theorem~11]{CS2011}
\label{thm: symmetric rank}
Let $h(x,y) = \sum_{i=0}^d h_i \binom{d}{i} x^i y^{d-i}$
be a binary form of degree $d$.
Then, we have the following:
\begin{itemize}
\item The symmetric border rank of $h$
is
\be
r \, := \, \rank \, C_{\lceil \frac{d}{2} \rceil, \lfloor \frac{d}{2} \rfloor }(h).
\ee

\item If $d$ is even and $r=d/2+1$,
then $\rank_S(h) = r$.

\item If $d$ is odd, or if $r<d/2+1$,
then $\rank_S (h) = r$ or $d- r + 2$
which can be determined as follows: let $ (f_0, f_1, \ldots, f_r) \ne 0$
be a vector from $\ker C_{d-r,r}(h)$,
which is unique up to scaling. If the polynomial
\[
f(x,y) \, :=\, f_0 x^r + f_1 x^{r-1}y + \cdots + f_r y^r
\]
has no multiple roots, then $\rank_S(h) = r$;
otherwise, $\rank_S(h) = d-r+2$.

%
%
%
%

\end{itemize}
\end{theorem}

Once we know the symmetric rank $k:=\rank_S(h)$,
Sylvester's method can be applied to compute
the Waring decomposition of the binary form $h(x,y)$.
By the above theorem, either $k=r$ or $k = d-r+2$.
Select a generic vector
$0 \ne (g_0, g_1, \ldots, g_k) \in \ker C_{d-k,k}(h)$.
Then, the binary form
\be \label{def:g(x,y)}
g(x,y) \, := \, g_0 x^k + g_1 x^{k-1}y + \cdots + g_k y^k
\ee
has $k$ distinct complex roots, say,
$(a_1, b_1), \ldots, (a_k, b_k)$,
in the projective space $\mathbb{P}^1$. Moreover,
there exist scalars $\lmd_1, \ldots, \lmd_k$ satisfying
\[
h(x,y) = \lmd_1 (a_1 x + b_1 y)^d + \cdots + \lmd_k (a_k x + b_k y)^d.
\]
The above is justified by the following theorem of Sylvester.

\begin{theorem}[\cite{Sylvester1851-1,Sylvester1851-2}]
\label{thm:Sylvester}
A binary form $h(x,y) = \sum_{i=0}^d h_i \binom{d}{i} x^i y^{d-i}$
of degree $d$ has the decomposition
$
h (x,y) = \sum_{i=1}^k \lambda_i (a_i x + b_i y)^d
$
for $\lambda_1, \ldots, \lmd_k \ne 0$ and
$(a_1, b_1), \dots, (a_k, b_k) \in \cpx^2$
pairwise linearly independent if and only if there exists
$ 0  \ne (g_0, g_1, \ldots, g_k) \in \ker  C_{d-k,k}(h)$
such that the binary form $g(x,y)$ as in \reff{def:g(x,y)}
has $(a_1, b_1), \dots, (a_k, b_k)$ as complex roots in $\mathbb{P}^1$.
%
%
\end{theorem}

\section{Vandermonde ranks and decompositions}
\label{sec:Vdcmp}
%
%
%
%
Although Hankel tensors are seemly very special symmetric tensors, every low rank symmetric tensor can be written as a Hankel tensor under a change of coordinate. More precisely, we have the following:
\begin{proposition}
For a generic symmetric tensor
$\mathcal{S} \in \operatorname{S}^m (\mathbb{C}^n)$ such that
$\operatorname{rank}_S(\mathcal{S}) \le n+2$,
there exists an $n\times n$ invertible matrix $G$ such that
\[
G \cdot \mathcal{S}\in \operatorname{H}^m(\mathbb{C}^n)~\text{and}~\rank_S(\mathcal{S})
= \rank_V (G \cdot \mathcal{S}),
\]
where $\cdot$ is the diagonal action.\footnote{
The diagonal action $G \cdot \mathcal{S}$ is defined as:
if $\mathcal{S} = \sum_i (u_i)^{\otimes m}$ is a decomposition, then
$G \cdot \mathcal{S} = \sum_i (Gu_i)^{\otimes m}$.
}
%
%
In particular, if $\mH\in \operatorname{H}^m(\mathbb{C}^n)$ is generic\footnote{$\mH$
is a generic element in the set of all Hankel tensors with $\rank_S(\mH)\le n+2$.}
such that $\rank_S(\mH)\le n+2$, then we must have $\rank_V(\mH) = \rank_S(\mH)$.
\end{proposition}
\begin{proof}
Since $\mathcal{S}$ has symmetric rank at most $(n+2)$, we can write $\mathcal{S}$ as
\[
\mathcal{S} = \sum_{j=1}^{r} u_j^{\otimes m},\quad r\le n+2.
\]
Since $\mathcal{S}$ is generic, we may assume that $u_1,\dots, u_{r}$
are in general position, i.e., any $k$ of them span
a linear subspace of dimension $\min\{k,n\}$ for each $k=1,\dots, r$.
By \cite[Theorem 1.18]{Harris1992}, there is a unique curve
$C\subset \mathbb{P}^{n-1}$ passing through $[u_1],\dots, [u_{n+2}]\in \mathbb{P}^{n-1}$,
which is projectively equivalent\footnote{
Two projective varieties $X,Y\subseteq \mathbb{P}^{n-1}$ are \emph{projectively equivalent}
if there exists an invertible $n \times n$ matrix $G$ such that $Y = G X$.}
to the rational normal curve $v_n(\mathbb{P}^1)$.
Therefore there exists an $n\times n$ invertible matrix $G$ such that
\[
G u_j  = (a_j^{n-1}, a_j^{n-2}b_j,\dots, a_j b_j^{n-2}, b_j^{n-1}),\quad (a_j,b_j) \in \mathbb{C}^2,\quad 1\le j \le n+2
\]
and hence $G \cdot \mathcal{S} \in \operatorname{H}^m(\mathbb{C}^n)$.
The second part follows easily from the first part.
\end{proof}

For a Hankel tensor $\mH \in \H^m( \cpx^n )$, let
$d = (n-1)m$ and $h = (h_0, h_1, \ldots, h_{d} )$ be the vector
as in \reff{eqn:defition of h}. We can think of $h$ as the
symmetric binary tensor in $\S^d (\cpx^2)$
that is represented by the binary form
\begin{equation}\label{eqn:binary}
h(x,y):= \sum_{i=0}^d h_i \binom{d}{i} x^i y^{d-i}.
\end{equation}
By writing $\rank_S(h)$ (resp., $\brank_S(h)$),
we mean the symmetric rank (resp., the symmetric border rank) of
the tensor represented by $h(x,y)$.
There is a one-to-one linear map between Hankel tensors and binary forms,
which is given as:
\begin{equation}\label{eqn:tensor-binary form}
\pi: \Hank^m(\mathbb{C}^n) \to \operatorname{S}^d(\mathbb{C}^2),
\quad  \mathcal{H} \mapsto h.
\end{equation}
In the above, $h$ is determined by $\mH$ as in \reff{eqn:defition of h}.
Clearly, the map $\pi$ is a bijection between
$\Hank^m(\mathbb{C}^n)$ and $\operatorname{S}^d(\mathbb{C}^2)$,
and $\rank_S(h) = 1$ if and only if $\rank_V (\mH)=1$.
A Vandermonde rank decomposition of $\mH$ is equivalent to
a symmetric rank decomposition of $h$.
We have the following basic lemma.

\begin{lemma}\label{lemma: V-rank = S-rank}
Let $\pi,\mH, h$ be as above. Then,
for all $\mH \in \Hank^m(\mathbb{C}^n)$, it holds that
\[
\rank_V(\mH) = \rank_S( h ),\quad
\brank_V(\mH) = \brank_S(h).
\]
\end{lemma}
\begin{proof}
%
%
We notice that $\pi$ is actually an isomorphism between two vector spaces. This implies in particular that $\pi$ is continuous if we equip both vector spaces the Euclidean topology. Hence it is sufficient to prove $\rank_V(\mH) = \rank_S(h)$ and the border version follows from the continuity of $\pi$. To this end, since $\pi$ is a bijective correspondence between rank one tensors, $\rank_V(\mH) = \rank_S(h)$ follows from the linearity of $\pi$.
\end{proof}

Lemma~\ref{lemma: V-rank = S-rank} establishes the equivalence
between Vandermonde decompositions of Hankel tensors
and symmetric decompositions of binary forms. Therefore,
all results about binary forms can be applied directly to
Hankel tensors via the map $\pi$. When a Hankel tensor $\mH$ is generic,
the symmetric binary tensor $h$ is also generic.
By Theorems~\ref{thm: symmetric rank}, \ref{thm:Sylvester}
and Lemma~\ref{lemma: V-rank = S-rank},
we can get the following corollary about
Vandermonde ranks of generic Hankel tensors.

\begin{corollary} \label{thm: V-rank}
If $\mathcal{H}\in \Hank^m(\mathbb{C}^n)$ is generic, then
\[
\rank_V(\mH)  \, = \, \left\lceil \frac{(n-1)m+1}{2} \right\rceil.
\]
\end{corollary}

The symmetric rank and the decomposition of a binary form is determined in
Theorem~\ref{thm: symmetric rank}. Consequently, by Lemma~\ref{lemma: V-rank = S-rank},
we can determine the Vandermonde rank decomposition
of a Hankel tensor by the following algorithm.

\begin{algorithm}
\label{algorithm: 1}
\rm
(Vandermonde rank decompositions for Hankel tensors,
motivated from Sylvester's algorithm for binary tensors.)
For a given $\mathcal{H}\in \Hank^m(\mathbb{C}^n)$,
let $d=(n-1)m$ and $h$ be as in \eqref{eqn:defition of h}.
Do the following:

\begin{enumerate}

%
%

\item[Step~$1$:] Let $s = \lfloor d/2 \rfloor$ and form the matrix
$C_{d-s,s}(h)$ as in \eqref{eqn:definition of C_{d-r,r}}.

\item[Step~$2$:] Let $r := \rank \, C_{d-s,s}(h)$. Find
$0 \ne (f_0, f_1, \ldots, f_r) \in \ker C_{d-r,r}(h)$.
Set $f(x,y) := f_0 x^r + f_1 x^{r-1}y  + \cdots + f_r y^r$.

\item[Step~$3$:] If $r = d/2 + 1$ or $f(x,y)$
has no multiple roots, then $\rank_V(\mH) = \rank_S(h)=r$;
otherwise, $\rank_V(\mH) = \rank_S(h)=d-r+2$.
%
%

\item[Step~$4$:] Let $k:= \rank_V(\mH)$, which is either $r$ or
$d-r+2$. If $k = r$, compute the $r$ distinct roots
$(a_1, b_1), \ldots, (a_r, b_r)$ of the binary form $f(x,y)$.
If $k=d-r+2$, select a generic
$0 \ne (g_0, g_1, \ldots, g_k) \in \ker C_{d-k,k}$
and compute the $k$ distinct roots
$(a_1, b_1), \ldots, (a_k, b_k)$ of the binary form
$
g(x, y) := g_0 x^k + g_1 x^{k-1} y + \cdots + g_k y^k.
$

\item[Step~$5$:] Determine the scalars $\lmd_1, \ldots, \lmd_k$ such that
\[
\mH = \sum_{i=1}^k \lmd_i (a_i^{n-1}, a_i^{n-2}(-b_i), \ldots, (-b_i)^{n-1})^{\otimes m}.
\]
The above is equivalent to
$h(x,y) = \sum_{i=1}^k \lmd_i (a_ix + b_iy)^d$.

\end{enumerate}
\end{algorithm}
Algorithm~\ref{algorithm: 1} is mathematically equivalent to 
computing rank decompositions of binary symmetric tensors, 
through the map $\pi$ defined in 
\eqref{eqn:tensor-binary form} and Lemma~\ref{lemma: V-rank = S-rank}.
It is orginated from Sylvester's algorithm \cite{Sylvester1851-1,Sylvester1851-2}
(also see \cite[Algorithm~1.1]{BCMT10}) for the case of simple roots
and the results of Comas and Seuger \cite{CS2011} for the case of multiple roots
(also see \cite[Algorithms~1,2]{BerGimIda11}).
%
%
For symmetric tensors of generic rank or subgeneric rank,
the uniqueness of the rank decompositions
is studied in \cite{ChOtVan15} and \cite{GalMel}.
By Theorems~\ref{thm: symmetric rank} and \ref{thm:Sylvester}
and Lemma~\ref{lemma: V-rank = S-rank},
we can get similar uniqueness results
about Vandermonde rank decompositions.

\begin{theorem} \label{Vdcmp:unique}
Let $\mH \in \H^m(\cpx^n)$ be a Hankel tensor.
In Algorithm~\ref{algorithm: 1},
if the rank of $C_{d-s,s}(h)$ is less than $d/2+1$ and the binary form
$f(x,y)$ has no multiple roots,
then the Vandermonde rank decomposition of $\mH$ is unique.
\end{theorem}

\subsection{Some examples}

In the following, we give examples to show how Algorithm~\ref{algorithm: 1} works.
We would like to remark that
\begin{enumerate}
\item[1)] It is possible that $\rank_V(\mH) > \rank_S(\mH)$.

\item[2)] The symmetric rank decomposition of $\mH$
is not necessarily a Vandermonde rank decomposition,
even if $\rank_V(\mH) = \rank_S(\mH)$.
\end{enumerate}
For each example, we not only find the Vandermonde decomposition of
the given Hankel tensor $\mH$, but also exhibit a symmetric decomposition of $\mH$.

\begin{example}
\label{example 1:rank_S. rank_V}
Consider the Hankel matrix $\mathcal{H}\in \Hank^2(\mathbb{C}^3)$:
\[
\mathcal{H} = (\mH_{ij})_{i,j =1}^3 =
\begin{bmatrix}
0 & 0 & 1\\
0 & 1 & 0 \\
1 & 0 & 0
\end{bmatrix},
\]
We have $d=4, s =2 $ and $h = (0,0,1,0,0)$. By Algorithm \ref{algorithm: 1},
we get $\rank_V (\mathcal{H}) = \rank_S (h) = 3$ and
\[
h (x,y)=   6 x^2 y^2 = \sum_{i=1}^ 3 \lambda_i (\alpha_i x + \beta_i y)^4,
\]
where (denote $i := \sqrt{-1}$)
\begin{align*}
(\lambda_1,\lambda_2,\lambda_3) &= (\frac{1}{3}, -\frac{1}{6} -\frac{\sqrt{3}}{6}i, -\frac{1}{6} +\frac{\sqrt{3}}{6}i), \\
(\alpha_1,\alpha_2,\alpha_3) &= (1,\frac{1+\sqrt{3}i}{2}, \frac{1-\sqrt{3}i}{2}), \\
(\beta_1,\beta_2,\beta_3) &= (1, -1,-1).
\end{align*}
So, $\mathcal{H}$ has the Vandermonde rank decomposition
\[
\mathcal{H} = \sum_{i=1}^3 \lambda_i (\beta_i^2 ,\alpha_i\beta_i,\alpha_i^2)^{\otimes 2}.
\]
On the other hand, $\mathcal{H}$ also has the symmetric rank decomposition
\begin{equation}\label{eqn:zx+y^2}
\mathcal{H} =- \frac{1}{2} (e_1 - e_3)^{\otimes 2}
 + \frac{1}{2} (e_1 + e_3)^{\otimes 2} + e_2^{\otimes 2},
\end{equation}
where $\{e_1,e_2,e_3\}$ is the standard unit basis of $\mathbb{C}^3$.
Indeed, $\rank_S(\mathcal{H}) =3$,
because $\mH$ is a matrix and all the tensor ranks are the same.
\end{example}

\begin{example}
\label{example 2:rank_S. rank_V}
Let $\mathcal{H}\in \Hank^3(\mathbb{C}^3)$ be the Hankel tensor such that
\[
\mH_{ijk} =\begin{cases}
1, \text{ if } i + j + k =7, \\
0,\text{ otherwise}.
\end{cases}
\]
The vector $h = (0,0,0,0,1,0).$
In Algorithm \ref{algorithm: 1}, $s=3$ and
$\rank\,C_{3,3}(h) =3$. The unique (up to scaling) vector from the null space of
$C_{3,3}(h)$ is $f = (1, 0 , 0, 0)$.
The equation $f(x,y) = x^3 = 0$ has a triple root. Therefore, we obtain $\rank_V(\mathcal{H}) = d - 3 + 2 = 5$. Moreover, the matrix $C_{1,5}$ is
\[
C_{1,5} = \begin{bmatrix}
0 & 0 & 0 & 0 & 1 & 0\\
0 & 0 & 0 & 1 & 0 & 0
\end{bmatrix}
\]
and $g = \begin{bmatrix}
1 &0 &0 &0& 0& -1
\end{bmatrix}^\mathsf{T} \in \operatorname{Ker}(C_{1,5})$. The binary form defined by $g$ is
\[
g(x,y) = x^5 - y^5 = \prod_{j=0}^4 (x - w^j y),\quad w = \exp(\frac{2\pi i}{5}).
\]
We may obtain the linear system
\[
\begin{bmatrix}
1 & w^5 & w^{10} & w^{15}  & w^{20} \\
5 & 5 w^4 & 5 w^8 & 5 w^{12} & 5 w^{16} \\
10 & 10 w^3 & 10 w^6 & 10 w^9 & 10 w^{12} \\
10 & 10 w^2 & 10 w^4 & 10 w^6 & 10 w^8 \\
5 & 5 w & 5 w^2 & 5 w^3 & 5 w^4  \\
1 & 1 & 1 & 1 & 1
\end{bmatrix} \begin{bmatrix}
 \lambda_0\\
 \lambda_1 \\
 \lambda_2 \\
 \lambda_3 \\
 \lambda_4
\end{bmatrix} = \begin{bmatrix}
0 \\
1 \\
0\\
0 \\
0 \\
0
\end{bmatrix}
\]
for $\lambda_j$'s by comparing coefficients of $x^4y$ and $\sum_{j=0}^4 \lambda_j (w^j x + y)^5$. It is clear that $ \begin{bmatrix}
1/5 & w^1/5 & w^2/5 & w^3/5 & w^4/5
\end{bmatrix}^\mathsf{T}$ is a solution hence we have
\[
h(x,y) = 5 x^4y = \sum_{j=0}^4 \frac{w^j}{5} (w^j x + y)^5
\]
and
\[
\mH = \sum_{j=0}^4 w^j (1,w^j,w^{2j})^{\otimes 3}.
\]

On the other hand, the polynomial associated to $\mH$ is $3(xz + y^2 ) z$,
which has the Waring decomposition:
\begin{equation}\label{eqn:(xz+y^2)z}
\frac{1}{2} \left[ -\frac{1}{2} (2x-z)^3 + \frac{8}{3} (x-z)^3
+ \frac{1}{6} (2x + z)^3 +(y+z)^3 - (y-z)^3 \right].
\end{equation}
%
So,
$\rank_S(\mathcal{H}) \le  5$. In fact, $\rank_S(\mathcal{H}) = 5$
by \cite[Table $1$]{LT2010}.
However, \eqref{eqn:(xz+y^2)z} is not a Vandermonde rank decomposition. In particular, this also implies that symmetric decompositions of $\mH$ are not unique.
\end{example}

\begin{example}
\label{x^{m-1}y + z^m}
Consider the Hankel tensor $\mH \in \H^m(\cpx^3)$ be such that
\[
\mH_{i_1 \ldots i_m} = \begin{cases}
1,\quad \text{ if } i_1 + \cdots + i_m = m +1 \text{ or } 3m \\
0, \quad \text{ otherwise}.
\end{cases}
\]
The polynomial associated to $\mH$ is $m x^{m-1}y + z^m$,
$d = 2m$ and
\[
h_{l} = \begin{cases}
1, \quad \text{ if }  l = 1 \text{ or }2m\\
0, \quad \text{ otherwise}.
\end{cases}
\]
One can check that $\rank\, C_{d-s,s}(h)  = 3$ and for
$f\in \ker\, C_{2m-3,3}(h)$
the polynomial $f(x,y)$ has a multiple root.
Hence we have $\rank_V(\mathcal{H})  = 2m - 1$.
On the other hand, by \cite[Theorem 10.2]{LT2010}, we know
$m \le \rank_S(\mathcal{H}) \le m+1$. Therefore,
if $m\ge 3$, we have $\rank_S(\mathcal{H})< \rank_V(\mathcal{H})$.
\end{example}

\section{Rank relations for general orders}
\label{sec:evenorder}

A Hankel tensor has the usual rank (i.e., the cp rank), symmetric rank, Vandermonde rank,
border rank and symmetric border rank.
This section studies relations among these various ranks for Hankel tensors.
We discus them for the even and odd order cases separately.
To do that, we first need to consider catalecticant matrices for Hankel tensors.

\subsection{Catalecticant matrices}
\label{ssc:flat}

For the order $m$, let $m_1 := \lceil m/2 \rceil.$
If $m=2m_0$ is even, $m_1 = m_0$;
if $m=2m_0+1$ is odd, $m_1 = m_0 + 1$.
A symmetric tensor can be flattened into catalecticant matrices \cite{IarKan99}.
Here, we consider the most square ones.
For each $\mH \in \Hank^m(\mathbb{C}^n)$,
denote by $\mbox{Flat}(\mH)$ the matrix, whose row is indexed by
an integral tuple $I=(i_1 \ldots i_{m_1})$
and whose column is indexed by another one
$J= (i_{m_1+1}, \ldots, i_{m})$,
such that the entries of $\mbox{Flat}(\mH)$ are given as
\[
\mbox{Flat}(\mH)_{I,J}  \, = \,
\mH_{i_1 \ldots i_{m_1}i_{m_1+1} \ldots i_{m}}.
\]
Note that $\mbox{Flat}(\mH)$ is a $n^{m_1}$-by-$n^{m-m_1}$ matrix.
There is a correspondence between catalecticant matrices and
symmetric flattenings for tensors \cite[p.76]{Landsberg2012}.
Because $\mH$ is Hankel, $\mbox{Flat}(\mH)$ has repeated rows and columns.
We consider a submatrix of $\mbox{Flat}(\mH)$ that does not have repeated ones.
Let $\mbox{F}(\mH)$ be the submatrix of $\mbox{Flat}(\mH)$
whose row index $I=(i_1, \ldots, i_{m_1})$ is such that
\[
i_1 \leq \cdots \leq i_{m_1}, \quad
m_1 \leq i_1 + \cdots + i_{m_1} \leq nm_1,
\]
and whose column index $J=(i_{m_1+1},\ldots, i_{m})$ is such that
\[
i_{m_1+1} \leq \cdots \leq i_{m}, \quad
m-m_1 \leq i_{m_1+1} + \cdots + i_{m} \leq n(m-m_1).
\]
In the following, we give an expression for $\mbox{F}(\mH)$.

\begin{lemma}  \label{lm:matF(H)}
Let $\mbox{Flat}(\mH)$, $\mbox{F}(\mH)$ be as above.
Let $l:=(n-1)m_0$. Then,
\be \label{form:F(H)}
\mbox{F}(\mH)
= C_{d-l,l}(h),
\ee
where $C_{d-l,l}(h)$ is defined as in \eqref{eqn:definition of C_{d-r,r}}.
Moreover, $\rank \, \mbox{F}(\mH) = \rank \, \mbox{Flat}(\mH)$.
\end{lemma}
\begin{proof}
If $\mH = (a^{n-1}, a^{n-2}b, \ldots, b^{n-1})^{\otimes m}$
is rank-$1$, then
\[
\mbox{F}(\mH) =
\bbm
a^{d} & a^{d-1}b &  a^{d-2} b^2 & \cdots & a^{d-l} b^{l} \\
a^{d-1}b &  a^{d-2} b_j^2 &  a^{d-3} b^3 & \cdots & a^{d-l-1} b^{l+1} \\
 \vdots &  \vdots &  \vdots & \ddots &  \vdots \\
a^{d-l}b^{l} &  a^{d-l-1} b^{l+1} & a^{d-l-2} b^{l+2} & \cdots &  b^{d} \\
\ebm.
\]
The equality~\reff{form:F(H)}is clearly true.
If $\mH$ is not rank-$1$, then $\mH$ is a sum of rank-$1$ Hankel tensors, say,
$
\mH = \mH_1 + \cdots + \mH_r,
$
where each $\mH_i$ is rank-$1$.
As a matrix-valued function, $\mbox{F}(\mH)$ is linear in $\mH$, so
\[
\mbox{F}(\mH) = \mbox{F}(\mH_1) + \cdots + \mbox{F}(\mH_r) = C_{d-l,l}(h).
\]
Hence, \reff{form:F(H)} is true.
Note that $\mbox{F}(\mH)$ is the maximum submatrix of
$\mbox{Flat}(\mH)$ that does not have any repeated rows and columns.
%
%
So, $\mbox{F}(\mH)$ and $\mbox{Flat}(\mH)$ have the same rank.
\end{proof}

\subsection{The even order case}

When the order $m=2m_0$ is even, the number $d := (n-1)m$ is also even.
So $s = \lfloor d/2 \rfloor = (n-1)m_0$.
Recall that $h$ is determined by $\mH$ as in \reff{eqn:defition of h} and
$C_{d-s,s}(h)$ is the Hankel matrix determined by
$h$ as in \eqref{eqn:definition of C_{d-r,r}}.
By Theorem \ref{thm: symmetric rank}, the rank $r$ of $ C_{d-s, s}$
is equal to the border rank of $h$,
which also equals the Vandermonde border rank of $\mH$.
So, it always holds that
\begin{equation}\label{eqn:border Vrank}
\brank_S (h) = \brank_V(\mH) = r.
\end{equation}
If $r=s+1$, then $\rank_V(\mH) = r$ by Theorem \ref{thm: symmetric rank}
and if $r < s+1$, then there exists a unique (up to scaling) vector
$0 \ne f \in \ker \,C_{d-r,r} (h)$.
The rank relations are summarized as follows.

\begin{theorem}\label{thm:Srank:bSrank}
Suppose $m=2m_0$ is even.
For $\mathcal{H}\in \Hank^m(\mathbb{C}^n)$,
let $h,d,s,r,f$ be as above.
Then, we have:
\bit

\item [(i)]\label{condition}
If $r = s + 1$, or if $r < s + 1$ and $f$ has no multiple roots, then
\be \label{ev:rk=:rSV:Hh}
\rank_V (\mc{H} )  =  \rank_S(\mc{H} )  =  \rank(\mH)  = \brank_V(\mc{H})
= \brank_S(\mathcal{H}) = \brank(\mH) = r,
\ee

\item [(ii)] If $r < s + 1$ and $f$ has a multiple root,
then
\be  \label{rkl:v>s=bbs=r}
d-r+2  =\rank_V (\mc{H})
\geq \rank_S(\mathcal{H}) \geq \brank_V(\mH) =
\brank_S(\mathcal{H}) = \brank(\mH) = r.
\ee

%
%

\eit

\end{theorem}

\begin{proof}
\noindent
(i) If $r = d/2 + 1$, or $r < d/2 + 1$ and $f$ has no multiple roots,
by Theorem~\ref{thm: symmetric rank} and
Lemma~\ref{lemma: V-rank = S-rank},
\[
r = \rank_S(h) = \rank_V(\mH) \geq \rank_S(\mH).
\]
It is well-known that the border rank of a tensor is always greater than
or equal to the rank of its flattening matrix \cite{Landsberg2012}.
Since $m=2m_0$ is even, $s=l=(n-1)m_0$.
By Lemma~\ref{lm:matF(H)}, we have
\[
\brank (\mH) \geq \rank\, \mbox{Flat}(\mH) = \rank\, \mbox{F}(\mH)
= \rank \, C_{d-s,s}(h) = r.
\]
Moreover, we also have
\[
\rank_S(\mH) \geq \rank(\mH) \geq \brank(\mH) \geq r,
\]
\[
\rank_S(\mH) \geq \brank_S(\mH) \geq \brank(\mH) \geq r.
\]
Since $r \geq \rank_S(\mH)$, all the ranks must be the same
and the equalities in \reff{ev:rk=:rSV:Hh} hold.

\medskip
\noindent
(ii) If $r < d/2 + 1$ and $f$ has a multiple root, then,
by Theorem~\ref{thm: symmetric rank} and
Lemma~\ref{lemma: V-rank = S-rank}, we have
\[
\rank_V(\mathcal{H}) = \rank_S(h)=d-r+2> r.
\]
Note that the symmetric border rank of $h$ is $r$,
by Theorem~\ref{thm: symmetric rank}.
Then, Lemma~\ref{lemma: V-rank = S-rank} implies that
\[
r = \brank_S (h) \ge \brank_{S}( \mc{H} )
\geq \brank ( \mc{H} ).
\]
As in the proof of (i), we can also prove that
\[
\brank_S(\mH) \geq \brank(\mH) \geq
\rank\, \mbox{Flat}(\mH) \geq  \rank\, \mbox{F}(\mH) = r.
\]
Hence, \reff{rkl:v>s=bbs=r} is true, because
$\rank_S(\mH) \geq \brank_S(\mH)$.
\end{proof}

%
%
%
%

Theorem~\ref{thm:Srank:bSrank} immediately implies the following.

\begin{corollary}\label{comon:true:even}
If $\mH \in \Hank^m(\mathbb{C}^n)$ is generic
and $m$ is even, then its cp rank, symmetric rank,
border rank, symmetric border rank, and Vandermonde rank
are the same, which is $1+(n-1)m/2$.
\end{corollary}
\begin{proof}
When $\mH$ is generic,
the vector $h$ is also generic
and $\rank \, C_{s,s}(h) = d/2+1$.
By Theorem \ref{thm:Srank:bSrank}(i),
all the ranks are equal.
\end{proof}

In particular, Theorem \ref{thm:Srank:bSrank}
also implies the following:

\begin{corollary}\label{cor:R_S = R_V}
For a generic Hankel tensor of an even order,
Algorithm~\ref{algorithm: 1} produces a Vandermonde decomposition
that achieves its cp rank, symmetric rank, border rank,
symmetric boarder rank and Vandermonde rank simultaneously.
\end{corollary}

In the following, we give some examples to show
applications of Theorem \ref{thm:Srank:bSrank}.
In particular, we would like to remark that:
\begin{itemize}

\item [1)] It is possible that
$\rank_V(\mathcal{H}) > \rank_S(\mathcal{H})$, even if the order is even.

\item[2)] We may have
$
\rank_V(\mathcal{H})  = \rank_S(\mathcal{H}) >  \brank_S(\mathcal{H}).
$
\end{itemize}

\begin{example}
%
%
Consider the Hankel tensor $\mH$ in Example~\ref{x^{m-1}y + z^m}.
We have
\[
\rank_V(\mathcal{H}) = 2m - 1, \quad
m \leq \rank_S(\mathcal{H}) \leq m+1, \quad \brank_S(\mathcal{H}) = 3.
\]
If $m \geq 4$, then
$\brank_S(\mathcal{H}) < \rank_S(\mathcal{H}) < \rank_V(\mathcal{H})$.
\end{example}

\begin{example}
\label{x^{m-2}y^2 + x^{m-1}z}
Consider the Hankel tensor
$\mathcal{H} \in \H^m(\cpx^3)$ such that
\[
\mH_{i_1,\dots, i_m} = \begin{cases}
1, \quad \text{ if }i_1 + \cdots + i_m = m+2\\
0,\quad \text{ otherwise}.
\end{cases}
\]
The polynomial associated to $\mH$ is
$\binom{m}{2}x^{m-2}y^2 + mx^{m-1}z$.
By Algorithm \ref{algorithm: 1} and Theorem \ref{thm:Srank:bSrank},
$\rank_V(\mathcal{H})  = 2m-1$.
By \cite[Theorem 10.2]{LT2010},
\[
\brank_S(\mathcal{H}) = 3, \quad
m \leq \rank_S(\mathcal{H}) \leq 2m-1.
\]
If $m \geq 4$, $\brank_S(\mathcal{H}) < \rank_S(\mathcal{H})$.
\end{example}

\subsection{The odd order case}
\label{ssc:oddord}

When the order $m=2m_0+1$ is odd,
the number $d = (n-1)m$ might not be even,
and $s = \lfloor d/2 \rfloor$ might be different from $(n-1)m_0$.
For $\mathcal{H}\in \Hank^m(\mathbb{C}^n)$,
$h$ is still the vector as in \eqref{eqn:defition of h}
and $r = \rank \, C_{d-s,s}(h)$.
Like \reff{eqn:border Vrank}, it also holds that
\[
\brank_S (h) = \brank_V(\mH) = r.
\]
Let $f = (f_0,\dots, f_r) \in \cpx^{r+1}$
be the unique vector (up to scaling) in $\ker C_{d-r,r}(h)$.
The rank relations for $\mH$ are as follows.

\begin{theorem}  \label{thm:rkrl:genodd} Let $n,m,d,\mH, h, r, f$ be as above.
Suppose $r \leq 1+(n-1)m_0$, then we have:
\bit

\item [(i)]
If $f(x,y)$ has no multiple roots, then
\be
\rank_V (\mc{H} ) = \rank_S(\mc{H} ) = \rank(\mH) = \brank_V(\mH) =
\brank_S(\mathcal{H}) = \brank(\mH) = r,
\ee

\item [(ii)] If $f(x,y)$ has a multiple root,
then
\be
d-r+2  =\rank_V (\mc{H})
\geq \rank_S(\mathcal{H}) \geq \brank_V(\mH) =
\brank_S(\mathcal{H}) = \brank(\mH) = r.
\ee

\eit

\end{theorem}

\begin{proof}
We follow the same approach as in the proof of
Theorem~\ref{thm:Srank:bSrank}. The difference is that
we need the assumption that $r \leq 1+(n-1)m_0$
when the order $m$ is odd.

\medskip
\noindent
(i) When $f$ has no multiple roots,
by Theorem \ref{thm: symmetric rank} and
Lemma~\ref{lemma: V-rank = S-rank}, we also have
\[
r = \rank_S(h) = \rank_V(\mH) \geq \rank_S(\mH).
\]
Let $l:= (n-1)m_0$.
In the following, we show that
\be \label{rk:C(h):lrs}
\rank \, C_{d-l,l}(h) \, = \, r.
\ee
Note that $r \leq 1 + l$,
$\rank \, C_{d-s,s}(h) = r$,
and $\brank_S(h)=r$ by Theorem \ref{thm: symmetric rank}.
\bit
\item When $r \leq l$,  \reff{rk:C(h):lrs} is true
by Proposition~9.7 of \cite{Harris1992}, since $\brank_S(h)=r$.

\item When $r = 1 + l$, we still have $\rank \, C_{d-l,l}(h) \leq r$.
If $\rank \, C_{d-l,l}(h) < r$, then $\brank_S(h)<r$
by Proposition~9.7 of \cite{Harris1992},
which is a contradiction. So,
\reff{rk:C(h):lrs} is also true.

\eit
Recall the matrices $\mbox{Flat}(\mH)$, $\mbox{F}(\mH)$
defined as in the subsection~\ref{ssc:flat}.
By Lemma~\ref{lm:matF(H)},
\[
\brank (\mH) \geq \rank\, \mbox{Flat}(\mH)
= \rank\, \mbox{F}(\mH) = \rank \, C_{d-l,l}(h) = r.
\]
Also note that $\brank(\mH) \leq \brank_S(\mH) \leq \rank_S(\mH)$.
Since $\rank_V(\mH)=r$, the relation \reff{rl:b<r<s<v}
and the above
imply that all the ranks must be the same.

\medskip

\noindent
(ii) The proof is the same as for
item (ii) of Theorem~\ref{thm:Srank:bSrank}.
\end{proof}

When $\mH$ is generic
and $m=2m_0+1$, for $n>2$, we have
\[
r = (n-1)m_0 +  \lceil  n/2 \rceil > 1 + (n-1)m_0.
\]
Hence, the rank relations in Theorem~\ref{thm:rkrl:genodd}
are not guaranteed any more. However,
we can still get a lower and upper bound for those ranks.

\begin{proposition}
If the order $m = 2m_0 +1$ is odd and $\mH \in \H^m (\cpx^n)$ is generic, then
\[
\rank_V(\mathcal{H}) =  m_0(n-1)  + \lceil  n/2 \rceil, \qquad \and
\]
\[
m_0(n-1) + 1 \le  \brank (\mH) \leq
\rank (\mathcal{H}) \le \rank_S(\mathcal{H})
\le \rank_V(\mathcal{H}).
\]
\end{proposition}
\begin{proof}
By Proposition~\ref{thm: V-rank}, we know that
$\rank_V(\mathcal{H}) =  m_0(n-1)  + \lceil  n/2 \rceil$.
The latter three inequalities are obvious.
It is enough to prove the first one.
We follow the proof of item (i) in Theorem~\ref{thm:rkrl:genodd}.
For all $\mH$, we always have ($l = (n-1)m_0$):
\[
\brank(\mH) \geq \rank \, \mbox{Flat} (\mH) \geq \rank\, \mbox{F}(\mH)
= C_{d-l, l}(h).
\]
When $\mH$ is generic, $h$ is also generic and
so $\rank\, C_{d-l, l}(h) = 1+l$,
which completes the proof.
\end{proof}

\section{Rank relations for order three}
\label{sec:order=3}

This section studies the relations
among various ranks of Hankel tensors
when the order $m=3$. For cubic Hankel tensors,
we are able to get better rank relations,
in addition to those given in Theorem~\ref{thm:rkrl:genodd}.
Recall that for each tensor $T$ and positive integer $p$, a linear map $\varphi_T^p$ is
defined in the subsection~\ref{ssc:mulinalg}
for vector spaces $A=B=C=\cpx^n$.
We use the standard unit vector basis
$\{ e_1, \ldots, e_n\}$ for $\cpx^n$ and we identify $\mathbb{C}^n$ with its dual so that $\{ e_1, \ldots, e_n\}$ is also a dual basis for itself. A tensor $T \in \cpx^n \otimes \cpx^n \otimes \cpx^n$ can be written as
\[
T = \sum_{i,j,k=1}^n t_{ijk} \, e_i \otimes e_j \otimes e_k,
\quad t_{ijk}\in \mathbb{C}.
\]
Let $1 \le p \le n-1$ be an integer. From \eqref{eqn:wedge} the linear map
\[
\varphi_T^p: \Big( \bigwedge^p \cpx^n \Big) \otimes \Big(\cpx^n\Big)^\ast
\to  \Big( \bigwedge^{p+1} \cpx^n \Big)  \otimes \cpx^n
\]
is defined by setting
\[
(e_{k_1}\wedge \cdots \wedge e_{k_p}) \otimes e_i \mapsto
\sum_{j,k=1}^n t_{ijk} \,
(e_{k_1}\wedge \cdots \wedge e_{k_p}\wedge e_j) \otimes e_k
\]
and extending it linearly.
By Theorem \ref{thm:lower bound}, we have
\[
\operatorname{brank}(T)  \ge
\rank\, \varphi_T^p \Big/ \binom{n-1}{p}.
\]
We construct the representing matrix $M := M_T^p$
for the linear map $\varphi_T^p$ under the standard basis. The set
\[
\Big\{(e_{k_1}\wedge \cdots \wedge e_{k_p}) \otimes e_i: \,
1\le k_1<\cdots < k_p\le n,1\le i \le n \Big \}
\]
is a basis of $\big( \bigwedge^p \cpx^n \big) \otimes (\cpx^n)^\ast$ and
\[
\Big \{(e_{k_1}\wedge \cdots \wedge e_{k_{p+1}}) \otimes e_k: \,
1\le k_1<\cdots < k_{p+1}\le n,1\le k \le n \Big \}
\]
is a basis of $\big( \bigwedge^{p+1} \cpx^n \big) \otimes \cpx^n$.
The representing matrix $M := M_T^p$ for $\varphi_T^p$ is
$n\binom{n}{p} \times  n\binom{n}{p+1}$.\footnote{
One can also take the transpose to obtain an
$n\binom{n}{p+1} \times n\binom{n}{p}$ matrix.
Since we only concern the rank, both matrices are okay for the proof.
}
We label the rows of $M$ by
\[
(J,i) \, := \, (j_1< \cdots< j_p, \,i)
\]
and label the columns of $M$ by
\[
(J',k)  \,:= \, (j'_1<\dots< j'_{p+1}, \, k).
\]
The entry of $M$ on the $(J,i)$-th row and $(J',k)$-th column is
\[
M_{(J,i),(J',k)} = \epsilon_{J,J',j} \, t_{ijk},
\]
where
\be  \label{notaton:dt:Jj}
\epsilon_{J,J',j} =
\begin{cases}
(-1)^{p-q},\quad  &\text{ if } j'_1 = j_1, \dots, j'_{q-1} = j_{q-1}, \mbox{ and}   \\
  & \quad\,  j'_{q} = j, j'_{q+1} = j_q, \dots, j'_{p+1} = j_p\\
0,\quad &\text{ otherwise.}
\end{cases}
\ee

The following is an example of $M_T^p$
when $T$ is a cubic Hankel tensor.

\begin{example}\label{ex: n=3}
Consider the linear map
$\varphi_\mH^1: \, \cpx^3 \otimes (\cpx^3)^\ast \to
\big( \wedge^2 \cpx^3 \big) \otimes \cpx^3$
for a Hankel tensor $\mH \in \operatorname{H}^3(\mathbb{C}^3)$. Note that
\[
\varphi_\mH^1 (e_j\otimes e_i) \, = \,
\sum_{j',k=1}^3 \mH_{ij'k} \, (e_j\wedge e_{j'}) \otimes e_k.
\]
Let $h$ be the vector as in \reff{eqn:defition of h}, then
\[
\arraycolsep=2.5pt
M_s =  \kbordermatrix{
& (e_1 \wedge e_2) \otimes e_1 & (e_1 \wedge e_2) \otimes e_2 & (e_1 \wedge e_2) \otimes e_3 & (e_1 \wedge e_3) \otimes e_1 & (e_1 \wedge e_3) \otimes e_2 & (e_1 \wedge e_3) \otimes e_3 & (e_2 \wedge e_3) \otimes e_1 &  (e_2 \wedge e_3) \otimes e_2 &  (e_2 \wedge e_3) \otimes e_3  \\
e_1 \otimes e_1 & h_1  & h_2 & h_3 & h_2 & h_3 & h_4 & 0  & 0  & 0\\
e_2 \otimes e_1 & -h_0 &-h_1&-h_2& 0   &  0   &  0   &h_2&h_3&h_4\\
e_3 \otimes e_1 & 0    &0     &0    & -h_0&-h_1&-h_2&-h_1&-h_2&-h_3\\
e_1 \otimes e_2 & h_2  & h_3 & h_4 & h_3 & h_4 & h_5 & 0  & 0  & 0\\
e_2 \otimes e_2 & -h_1&-h_2&-h_3& 0   &  0   &  0   &h_3&h_4&h_5\\
e_3 \otimes e_2 & 0    &0     &0    & -h_1&-h_2&-h_3&-h_2&-h_3&-h_4\\
e_1 \otimes e_3 & h_3  & h_4 & h_5 & h_4 & h_5  & h_6 & 0  & 0  & 0\\
e_2 \otimes e_3 & -h_2&-h_3&-h_4& 0   &  0   &  0   &h_4 &h_5 &h_6\\
e_3 \otimes e_3 & 0    &0     &0    & -h_2&-h_3&-h_4&-h_3&-h_4&-h_5\\
}.
\]
One can verify that $\rank\, M^1_\mH = 8$ when $\mH$ is generic.
Indeed, the sum of the third and the seventh column is equal to
the fifth column of $M_\mH^1$, which implies that $\rank M_\mH^1\le 8$.
Moreover, it is easy to verify that the submatrix obtained
by removing the seventh column from $M_\mH^1$ has full rank $8$.
Indeed, after a permutation on rows, the matrix
$M^1_\mH$ corresponds to the matrix in (3.8.1) of \cite[page 81]{Landsberg2012}
and the Koszul flattening matrix in \cite[Example~4.2]{OedOtt13}.
By Theorem \ref{thm:lower bound}, we can get $\brank(\mH)\ge 4$.
If $\varphi_\mH^2$ is used, we get another lower bound for $\brank(\mH)$.
However, $\rank\, M_\mH^2 \le 3$,
which is worse than the one by using $\varphi_\mH^1$.
\end{example}

\begin{theorem}\label{lemma:lower bound of Hankel tensors}
For a generic Hankel tensor $H\in S^3(\mathbb{C}^n)$,
its border rank is at least $\lfloor \frac{3n-1}{2} \rfloor$.
\end{theorem}
\begin{proof}
Let $r := \lfloor \frac{3n-1}{2} \rfloor$ and
$p := \lfloor \frac{n}{2} \rfloor$. By Theorem~\ref{thm:lower bound},
it is sufficient to prove that
the rank of the linear map $\varphi_\mH^p: \big( \bigwedge^p \cpx^n \big) \otimes (\cpx^n)^\ast
\to \big( \bigwedge^{p+1} \cpx^n \big) \otimes \cpx^n$
has rank $\binom{n-1}{p}r$.
By the lower semi-continuity\footnote{
To be more precise,
if  $M(x_1,\dots, x_k)$ is an $m\times n$ matrix polynomial
and $\rank\,M(a_1,\dots,a_k) = r$
for some $(a_1,\dots, a_k)\in \mathbb{C}^k$,
then there exists a Zariki open dense subset $U$
of $\mathbb{C}^k$ such that $\rank M(b_1,\dots, b_k) \ge r$
for all $(b_1,\dots, b_k)\in U$.
}
of the matrix rank function,
it is sufficient to prove that $\rank (M_\mH^p) \ge \binom{n-1}{p}r$
for some order three Hankel tensor $\mH$.
To do this, we take $\mH = (\mH_{ijk})$, where
\[
\mH_{ijk} = \begin{cases}
1,\quad \text{if}\quad i+j+k -2 = r, \\
0,\quad \text{otherwise}.
\end{cases}
\]
The tensor
$(e_{j_1}\wedge\cdots \wedge e_{j_p}) \otimes e_i $ is mapped by $\varphi_\mH^p$ to
\begin{equation}\label{eqn:defn of vaphi_H^p}
\sum_{j=1}^n (-1)^{p-q} (e_{j_1}\wedge \cdots \wedge e_{j_q} \wedge e_j
\wedge e_{j_{q+1}} \wedge \cdots \wedge e_{j_p}) \otimes e_{r +2  - i - j},
\end{equation}
where the summation is over
\[
1\le j_1 < \cdots < j_q < j  < j_{q+1} < \cdots < j_p \le n.
\]
We set $e_{r + 2 -i-j} = 0$ if $i + j   \ge r  + 2$ or $ i + j \le  \lfloor \frac{n+1}{2} \rfloor$.
Hence, the summand in \eqref{eqn:defn of vaphi_H^p} is non-zero if and only if
\[
\max \{1,r  - i - n + 2\} \le j \le \min \{n,r - i + 1\} \quad \text{and} \quad
j \not\in \{ j_1,\dots, j_p\}.
\]
%
%
For given $j_1 < \cdots < j_q < j < j_{q+1} <\cdots < j_p$, we let
\[
J := (j_1 < \cdots < j_p)~\text{and}~
J' := (j_1 < \cdots < j_q < j < j_{q+1} <\cdots < j_p).
\]
By \eqref{eqn:defn of vaphi_H^p}, the matrix $M_\mH^p$
has a block that is
\[
T_{J,J'}  \coloneqq  \epsilon_{J,J'} T_{j - \lfloor \frac{n+1}{2} \rfloor} =
\]
\[
\arraycolsep=2.5pt
 \kbordermatrix{
       & (J',n) & (J',n-1) & \cdots  & (J',r-j+2) & (J',r-j + 1)    & (J',r-j )    & \cdots & (J',1) \\
(J,1) & 0     & 0          & \cdots  & 0        & (-1)^{p-q} & 0               & \cdots & 0\\
(J,2) & 0     & 0          & \cdots  & 0        & 0               & (-1)^{p-q} & \cdots & 0\\
\vdots  & \vdots     & \vdots          & \ddots  & \vdots        & \vdots             & \vdots & \ddots & \vdots\\
(J,r-j+1) & 0     & 0          & \cdots  & 0        & 0               & 0 & \cdots & (-1)^{p-q}\\
\vdots  & \vdots     & \vdots          & \ddots  & \vdots        & \vdots             & \vdots & \ddots & \vdots\\
(J,n-1) & 0     & 0          & \cdots  & 0        & 0               & 0& \cdots & 0\\
(J,n) & 0     & 0          & \cdots  & 0        & 0               &0 & \cdots & 0\\
},
\]
where $\epsilon_{J,J'} = (-1)^{p-q}$ and $T_{j - \lfloor \frac{n+1}{2} \rfloor}$ is the $n\times n$ Toeplitz matrix whose $(i,k)$-th entry $t_{ik}$ is defined by
\[
t_{ik} = \begin{cases}
1, \quad \text{ if } k-i = j-\lfloor \frac{n+1}{2} \rfloor, \\
0,\quad \text{ otherwise}.
\end{cases}
\]
Clearly, we have
\[
\rank (T_{J,J'}) = \begin{cases}
r - j +1,   &\text{if}\quad  r - j +1 \le n,\\
2n- (r - j  + 1), \quad & \text{otherwise}.
\end{cases}
\]
In particular, if $J' = J \cup \{ \lfloor \frac{n+1}{2} \rfloor\}$, then $\rank (T_{J,J'}) = n$, i.e., the matrix $T_{J,J'}$ is the identity matrix up to a sign. Let $C_0$ be the set of sequences $J' = (j_1< \cdots < j_{p+1})$ such that $\lfloor \frac{n+1}{2} \rfloor \in J'$ and let $R_0$ be the set of sequences $J = (j_1<\cdots < j_p)$ such that $\lfloor \frac{n+1}{2} \rfloor  \not\in  J$. For each $J\in R_0$, there is a unique $J'\in C_0$ such that $J' = J \cup \{\lfloor \frac{n+1}{2} \rfloor \}$. We also let $C$ be the set of sequences $J' = (j_1< \cdots  < j_{p+1})$ such that $ \lfloor \frac{n+1}{2} \rfloor \not\in J'$ and let $R$ be the set of sequences $J = (j_1<\cdots < j_p)$ such that $\lfloor \frac{n+1}{2} \rfloor  \in  J$. Then the matrix $M_{\mH}^p$ may be visualized as follows:
\[
\arraycolsep=2.5pt
M_{\mH}^p = \kbordermatrix{
       & C_0 & C  \\
R_0 & D & T_1 \\
R     & T_2 & 0 \\
},
\]
where $D$ is the submatrix of $M_\mH^p$ obtained by taking $J$'s from $R_0$ and $J'$'s from $C_0$ and $T_1,T_2$ are defined in the same way. Since for each $J\in R_0$ there exists the unique $J'\in C_0$ such that $J' = J\cup \{ \lfloor \frac{n+1}{2} \rfloor \}$, we see that $D$ is actually a block diagonal matrix where each diagonal block is of the form $T_{J,J'}$, which is the $ n\times n $ identity matrix (up to a sign) because $J' = J\cup \{ \lfloor \frac{n+1}{2} \rfloor \}$. Indeed, $D$ is a diagonal matrix whose diagonal entries are either $1$ or $-1$. Moreover, the cardinality of both $C_0$ and $R_0$ is equal to
\[
\# C_0 = \# R_0 = \binom{n-1}{p},
\]
this implies that $D$ is a full rank $\binom{n-1}{p}n \times \binom{n-1}{p}n$ matrix.

Next, we apply column operations to $M_\mH^p$ to make it a triangular matrix. More precisely, we compute
\[
M_{\mH}^p \begin{bmatrix}
\Id_{\binom{n-1}{p}n}    & -D^{-1} T_1 \\
0     &  \Id_{\binom{n-1}{p+1}n}
\end{bmatrix} = \begin{bmatrix}
D & 0 \\
T_2 & -T_2 D^{-1} T_1
\end{bmatrix} = \begin{bmatrix}
D & 0 \\
T_2 & -T_2 D T_1
\end{bmatrix}.
\]
Here $\Id_m$ is the $m\times m$ identity matrix and the second equality follows from the fact that $D$ is a diagonal matrix whose diagonal entries are $1$ or $-1$. We denote by $M$ the matrix $-T_2 D T_1$ and by Lemma \ref{lemma: technical lemma}, we see that
\[
\rank M = \binom{n-1}{p+1}(p+1).
\]
Therefore, $M_\mH^p$ has the rank
\[
\rank M_\mH^p = \binom{n-1}{p}n + \binom{n-1}{p+1}(p+1) = \binom{n-1}{p}r.
\]
\end{proof}

\begin{lemma}\label{lemma: technical lemma}
Let $\mathcal{H}, M_\mathcal{H}^p,T_1,T_2,D$ and $M$ be as in the proof of Theorem \ref{lemma:lower bound of Hankel tensors}. The rank of $M$ is equal to $\binom{n-1}{p+1}(p+1)$.
\end{lemma}
The proof of Lemma~\ref{lemma: technical lemma}
will be given in the Appendix.

\begin{theorem}\label{thm: rank order three Hankel}
For a generic Hankel tensor $\mathcal{H}\in \Hank^3(\mathbb{C}^n)$, we have
\be \label{cuibc:allrank=}
\operatorname{brank}(\mathcal{H}) = \brank_S (\mathcal{H}) = \operatorname{rank}(\mathcal{H}) = \operatorname{rank}_S( \mathcal{H})= \rank_V( \mathcal{H}) = \left \lfloor \frac{3n-1}{2} \right \rfloor
\ee
\end{theorem}
\begin{proof}
By Theorem~\ref{thm: V-rank}, we have seen that $\rank_V(\mathcal{H})$ is
\[
\left \lceil \frac{3n-2}{2}  \right \rceil =
\left \lceil \frac{3n}{2}  \right \rceil -1 =
\left \lfloor \frac{3n-1}{2} \right \rfloor.
\]
Theorem~\ref{lemma:lower bound of Hankel tensors} implies that
$\brank(\mH) \geq \lfloor \frac{3n-1}{2} \rfloor $ when $\mH$ is generic.
Since we always have
\[
\operatorname{brank}(\mathcal{H}) \le \operatorname{rank}(\mathcal{H})
\le \operatorname{rank}_S( \mathcal{H}) \le \rank_V( \mathcal{H}),
\]
\[
\brank(\mathcal{H}) \le \brank_S(\mathcal{H}) \le \rank_S(\mathcal{H}),
\]
the conclusion follows directly.
\end{proof}

%
%

Theorem~\ref{thm: rank order three Hankel}
clearly implies the following.

\begin{corollary}
For a generic Hankel tensor of order three,
Algorithm \ref{algorithm: 1} gives a decomposition
that achieves its cp rank, symmetric rank,
border rank, symmetric border rank, and Vandermonde rank.
\end{corollary}

In Theorem~\ref{thm: rank order three Hankel},
we can get concrete conditions for the equalities there to hold.
In Algorithm~\ref{algorithm: 1}, by \reff{rl:b<r<s<v}
and Theorem~\ref{thm:lower bound},
we know \reff{cuibc:allrank=} holds if
\bit
\item $\rank M_\mH^p \geq \lfloor (3n-1)/2 \rfloor \binom{n-1}{p}$;
\item when $n$ is odd, $\rank\, C_{d-s,s}(h) = 1 + s$;
\item when $n$ is even, $\rank \, C_{d-s,s}(h) = 1+s$ and
the binary form $f(x,y)$ has no multiple roots.

\eit

In the following, we give some examples that
the conclusion of Theorem~\ref{thm: rank order three Hankel}
may not hold for {\it non-generic} Hankel tensors.

\begin{example}
Consider the Hankel tensor $\mathcal{H} \in \H^3(\cpx^3)$ such that
\[
\mH_{ijk} = \begin{cases}
1, \quad \text{ if }i + j + k = 8, \\
0,\quad \text{ otherwise}.
\end{cases}
\]
The polynomial associated to $\mH$ is $3yz^2$. We have $d=6,s=3$,
$h = (0,0,0,0,0,1,0)$ and $r = \rank C_{d-s,s} = 2$.
By Algorithm \ref{algorithm: 1}, we get
$\rank_V (\mathcal{H}) = 6$. However, $\rank_S(\mathcal{H}) = 3$
and $\brank_S(\mathcal{H}) = 2$ by \cite[Table $1$]{LT2010}
or \cite[Theorem 1.1 \& 1.2]{Oeding2016}.
However,
\[
\frac{3n-1}{2}   = 4 > \rank (\varphi_\mathcal{H}^1) = 2.
\]
Hence, $\brank(\mathcal{H}) =2$ and $\rank(\mathcal{H}) = 2$ or $3$.
In fact, $\rank(\mathcal{H}) = 3$.
If otherwise $\rank(\mathcal{H}) =2 $, then $\mathcal{H}$ has a decomposition
\[
\mathcal{H} = u_1 \otimes v_1 \otimes w_1 + u_1 \otimes v_2 \otimes w_2, \quad u_i,v_i,w_i\in \mathbb{C}^3,i=1,2.
\]
One may use Macaulay2 \cite{M2} to verify that
such a decomposition does not exist.
\end{example}

\begin{example}
Let $\mathcal{H}$ be the Hankel tensor as in
Example~\ref{x^{m-1}y + z^m} for $m =3$.
We know $\rank_V(\mathcal{H}) = 5$.
By \cite[Table~1]{LT2010},
\[
\rank_S(\mathcal{H}) = 4, \quad \brank_S(\mathcal{H}) = 3.
\]
Moreover, $\rank(\varphi^1_\mathcal{H}) = 3$, hence
\[
\brank(\mathcal{H}) = 3,\quad \rank(\mathcal{H}) = 3 \text{ or }4.
\]
Since the monomials $x^2y$ and $z^3$ do not share a common variable,
by \cite[Sec.~9.1.4]{Landsberg2012}, we have
\[
\rank (\mathcal{H}) = \rank (x^2 y) + \rank (z^3) = 3 +1 =4 = \rank_S(\mathcal{H}).
\]
\end{example}

\begin{example}
Consider the special case of Example~\ref{x^{m-2}y^2 + x^{m-1}z}
with $m = 3$. We have seen that
$\rank_V(\mathcal{H})  = 5$, $\brank_S(\mathcal{H}) = 3$.
By \cite[Theorem 10.2]{LT2010}\footnote{
\cite[Theorem 10.2]{LT2010} states that $m\le \rank_S(\mathcal{H}) \le 2m-1$ in general,
but by the remark after Theorem 10.2, $\rank_S(\mathcal{H})$
attains the upper bound $5$ if $m=3$.
}, $\rank_S(\mathcal{H}) =5$. One can check that $\rank (\varphi_\mathcal{H}^1) = 3$, hence
\[
3 =\brank(\mathcal{H}) = \brank_S(\mathcal{H}) <  \rank_V(\mathcal{H}) =5.
\]
This implies that $\rank(\mathcal{H}) =3,4$ or $5$.
Indeed, we may verify again by Macaulay2 \cite{M2} that
$
\rank(\mathcal{H}) = 5 = \rank_S(\mathcal{H}).
$
\end{example}

\section{Conclusions and Questions}
\label{sec:openqs}

The main results of this article are:

\begin{enumerate}

\item [1)] We give an algorithm (Algorithm \ref{algorithm: 1})
for computing Vandermonde rank decompositions for all Hankel tensors. In particular,
the Vandermonde rank of a generic
$\mathcal{H} \in \Hank^m(\mathbb{C}^n)$
is $\lceil \frac{m(n-1)+1}{2} \rceil $
(Proposition~\ref{thm: V-rank}).

\item [2)] We can determine the cp rank, symmetric rank,
border rank and symmetric border rank
of a Hankel tensor, under some concrete conditions
(Theorem \ref{thm:Srank:bSrank}, Theorem \ref{thm:rkrl:genodd}).

\item [3)] We prove that the cp rank, symmetric rank,
border rank, symmetric border rank and Vandermonde rank are all the same
for a generic Hankel tensor of order even or three
(Corollary~\ref{comon:true:even},
Theorem~\ref{thm: rank order three Hankel}).

\end{enumerate}

However, we do not know much about the rank relations for
generic Hankel tensors of odd order $m \geq 5$.
Naturally, we pose the following question:

\begin{question}\label{quest6.1}
For an odd order $m\ge 5$ and for a generic
Hankel tensor $\mH \in \H^m(\cpx^n)$, do we have
\[
\rank (\mH) = \rank_S(\mH) = \brank (\mH) = \brank_S(\mH) = \rank_V(\mH)?
\]
\end{question}

We point out that the answer to Question~\ref{quest6.1}
is ``no" if we replace ``generic" by ``all",
as we have already seen in the earlier examples.
However, we conjecture that
the answer to Question~\ref{quest6.1} is yes.

\begin{conjecture}\label{conjecture1}
The answer to Question~\ref{quest6.1} is yes.
\end{conjecture}

Finally, we conjecture that Comon's conjecture remains true
at least for Hankel tensors, although Y.~Shitov provides a counterexample in \cite{Shitov2017} which implies that Comon's conjecture does not hold for all symmetric tensors.

\begin{conjecture} \label{conj:CH}
For all $\mH \in \H^m(\cpx^n)$,
$\rank (\mH) = \rank_S(\mH)$.
\end{conjecture}

\bigskip
\noindent
{\bf Acknowledgement}\,
The authors would like to thank Lek-Heng Lim
and two anonymous referees for their fruitful suggestions
on improving the paper.
Jiawang Nie was partially supported by the NSF grants
DMS-1417985 and DMS-1619973. Ke Ye was partially supported
by National Key R\&D Program of China Grant no.~2018YFA0306702, the NSFC grant NSFC No.11688101, the Hundred Talents Program
of the Chinese Academy of Sciences
as well as the Recruitment Program of Global Experts of China.

%
%

\bibliographystyle{plain}

\appendix
\section{The proof of Lemma~\ref{lemma: technical lemma}}
\label{sec:appendix}
In this Appendix, we give the proof of
Lemma~\ref{lemma: technical lemma}, which is used in the proof of
Theorem~\ref{lemma:lower bound of Hankel tensors}.
We will also work out some examples to illustrate the idea of the proof. Readers are recommended to read these examples to better understand the proof. We first briefly describe the strategy we employ to prove Lemma~\ref{lemma: technical lemma}.
\begin{itemize}
\item first we investigate entries of $M$ to conclude that $M$ is a block diagonal matrix with blocks $M_0,\dots, M_p$ where $M_s$ has $\binom{n-p-1}{s} \binom{p}{s-1}(p+1)$ columns. This is done in \emph{Step} 1--3 below;
\item then we prove that each $M_s$ is of full rank by showing that the reduction of $M_s$ from $\mathbb{Z}$ to $\mathbb{Z}_2$ is non-singular. This is done in \emph{Step} 4 and 5 below.
\end{itemize}

\begin{proof}[Proof of Lemma~\ref{lemma: technical lemma}]
By definition of $R$ and $C$, we have
\[
\# R = \binom{n-1}{p-1} \ge  \# C = \binom{n-1}{p+1}.
\]
Then $D$ is a $\binom{n-1}{p}n\times \binom{n-1}{p}n$ matrix, $T_1$ is a $\binom{n-1}{p}n \times \binom{n-1}{p+1}n$ matrix and $T_2$ is a $\binom{n-1}{p-1}n \times \binom{n-1}{p}n$ matrix. This implies that $M = - T_2 D T_1$   is a $\binom{n-1}{p-1}n \times \binom{n-1}{p+1}n$. For each $J'\in C$ and $1\le k \le n$, we denote by $v_{J',k}$ the $(J',k)$-th column vector of $M$. We will prove that the matrix $M$ has rank $\binom{n-1}{p+1} (p+1)$ in the following steps.

\emph{Step $1$:} We describe blocks $M_{J,J'}$ of the matrix $M$. We may partition $T_1$ by blocks of size $n\times n$ and index them by elements in $R_0$ and $C$. To be more precise, for each $J\in R_0$ and $J'\in C$ we denote by $T_{1,J,J'}$ the submatrix obtained by taking rows $(J,1),\cdots,(J,n)$ and columns $(J',1),\dots, (J',n)$. Similarly, we may also partition $T_2$ (resp. $D$) by blocks of size $n\times n$ and index them by elements in $R$ (resp. $R_0$) and $C_0$ (resp. $C_0$). We denote these blocks by $T_{2,J,J'},J\in R,J'\in C_0$ (resp. $D_{J',J},J'\in C_0,J\in R_0$)\footnote{Here we remark that according to the definition of $D$, we should denote each block by $D_{J,J'}$ where $J\in R_0, J'\in C_0$, but we switch $J$ with $J'$ to simplify our notation. This is valid as $\# R_0 = \# C_0$ and $D^\mathsf{T} = D$.}. Since $M = - T_2 D T_1$, we may partition $M$ in the same fashion and denote these blocks by $M_{J,J'},J\in R, J'\in C$. Moreover,
\[
M_{J,J'} = - \sum_{J_0'\in C_0,J_0\in R_0}T_{2,J,J'_0} D_{J'_0,J_0} T_{1,J_0,J'}.
\]
We notice that
\begin{itemize}
\item $D_{J'_0,J_0}\ne 0$  if and only if $J'_0 = J_0 \cup\{\lfloor (n+1)/2 \rfloor\}, J'_0\in C_0, J_0 \in R_0$ .
\item $T_{2,J,J'_0}\ne 0$ only if $J\subsetneq J'_0$, $J\in R, J'_0\in C_0$.
\item $T_{1,J_0,J'}\ne 0$ only if $J_0 \subsetneq J', J_0\in R_0, J'\in C$.
\end{itemize}
Therefore, $M_{J,J'}\ne 0$ only if there exists $j < \lfloor (n+1)/2 \rfloor < j'$ such that
\[
J' = J\setminus \{\lfloor (n+1)/2 \rfloor\} \cup \{j,j'\},
\]
and
\begin{align*}
M_{J,J'}  = &-(T_{2,J,J\cup \{j\}} D_{J\cup \{j\}, J'\setminus \{j'\}} T_{1,J'\setminus \{j'\},J'} \\
&+ T_{2,J,J\cup \{j'\}} D_{J\cup \{j'\}, J'\setminus \{j\}} T_{1,J'\setminus \{j\},J'}).
\end{align*}
On the other hand, we recall from the proof of Theorem \ref{lemma:lower bound of Hankel tensors} that
\[
D_{J\cup\lfloor (n+1)/2 \rfloor,J} = \epsilon_{J,J\cup\{\lfloor (n+1)/2 \rfloor\}} \Id_n,\quad
T_{s,J,J\cup \{j\}} = \epsilon_{J,J\cup \{j\}} T_{j- \lfloor (n+1)/2 \rfloor}, \quad s= 1,2.
\]
Here $T_{l}$ is the Toeplitz matrix $(t_{ij})$ defined by
\[
t_{ij} = \begin{cases}
1,\quad \text{if } j-i = l,\\
0, \quad \text{otherwise}.
\end{cases}
\]
Thus we obtain
\begin{align*}
T_{2,J,J\cup \{j\}} &= \epsilon_{J,J\cup \{j\}} T_{j-\lfloor (n+1)/2 \rfloor}, \\
D_{J\cup \{j\},J'\setminus \{j'\}} &= \epsilon_{J'\setminus \{j'\}, J'\setminus \{j'\} \cup \{\lfloor (n+1)/2 \rfloor\}} \Id_n, \\
T_{1,J'\setminus \{j'\},J'} &= \epsilon_{J'\setminus \{j'\},J'} T_{j' - \lfloor (n+1)/2 \rfloor}, \\
T_{2,J,J\cup \{j'\}} &= \epsilon_{J,J\cup \{j'\}} T_{j'-\lfloor (n+1)/2 \rfloor}, \\
D_{J\cup \{j'\},J'\setminus \{j\}} &= \epsilon_{J'\setminus \{j\}, J'\setminus \{j\} \cup \{\lfloor (n+1)/2 \rfloor\}} \Id_n, \\
T_{1,J'\setminus \{j\},J'} &= \epsilon_{J'\setminus \{j\},J'} T_{j - \lfloor (n+1)/2 \rfloor}.
\end{align*}
We set
\begin{align*}
\delta_{J,J'} &= \delta_{J,j,j'} = \epsilon_{J,J\cup \{j\}}\epsilon_{J'\setminus \{j'\}, J'\setminus \{j'\} \cup \{\lfloor (n+1)/2 \rfloor\}} \epsilon_{J'\setminus \{j'\},J'} \\
\delta_{J,j',j} &= \epsilon_{J,J\cup \{j'\}}\epsilon_{J'\setminus \{j\}, J'\setminus \{j\} \cup \{\lfloor (n+1)/2 \rfloor\}} \epsilon_{J'\setminus \{j\},J'}
\end{align*}
and it is straightforward to verify that $\delta_{J,j',j} = - \delta_{J,j,j'}$. Hence we may write
\begin{equation}\label{eqn:blocks of M}
M_{J,J'} =  \delta_{J,J'}(T_{j'-\lfloor (n+1)/2 \rfloor} T_{j-\lfloor (n+1)/2 \rfloor} - T_{j-\lfloor (n+1)/2 \rfloor} T_{j'-\lfloor (n+1)/2 \rfloor} ),
\end{equation}
if $j < r-n + 1 < j'$ and $J' = J\setminus \{\lfloor (n+1)/2 \rfloor\} \cup \{j,j'\}$ and $M_{J,J'}=0$ otherwise.

\emph{Step $2$:} If $j < r-n + 1 < j'$ and $J' = J\setminus \{\lfloor (n+1)/2 \rfloor\} \cup \{j,j'\}$, then by \eqref{eqn:blocks of M}, the $(i,k)$-th entry $(M_{J,J'})_{i,k}$ of the $n \times n$ matrix $M_{J,J'}$ is zero unless $k-i = j+j'-2\lfloor (n+1)/2 \rfloor$ and in this case, we have:
\begin{itemize}
\item if $j+j'-2\lfloor (n+1)/2 \rfloor \ge 0$, then
\begin{equation}\label{eqn: entries-1}
(M_{J,J'})_{i,k} = \begin{cases}
\delta_{J,J'},&\text{ if } j'-\lfloor (n+1)/2 \rfloor \ge k \ge j+j'-2\lfloor (n+1)/2 \rfloor+1,  \\
-\delta_{J,J'},&\text{ if } n \ge k \ge 2n - r +j, \\
0, & \text{ otherwise.}
\end{cases}
\end{equation}
\item if $j+j'-2\lfloor (n+1)/2 \rfloor < 0$, then
\begin{equation}\label{eqn: entries-2}
(M_{J,J'})_{i,k} = \begin{cases}
\delta_{J,J'},&\text{ if } j'-\lfloor (n+1)/2 \rfloor \ge k \ge 1,\\
-\delta_{J,J'}, & \text{ if } j+j'-2r+3n-2 \ge k \ge j-r+2n, \\
0, & \text{ otherwise.}
\end{cases}
\end{equation}
\end{itemize}
In particular, we have
\begin{equation}\label{eqn: vanishing of middle columns}
(M_{J,J'})_{i,\lfloor \frac{n+1}{2} \rfloor } = (M_{J,J'})_{\lfloor \frac{n+1}{2} \rfloor ,k} = 0
\end{equation}
for any $J\in R,J'\in C$ and $1\le i \le n$.
By \eqref{eqn: entries-1} and \eqref{eqn: entries-2}, we see that $(M_{J,J'})_{i,k} = 0$ unless
\begin{itemize}
\item $J' = J\setminus\{\lfloor \frac{n+1}{2} \rfloor \} \cup \{j,j'\},  j < \lfloor \frac{n+1}{2} \rfloor  < j'$,
\item $k - i =  (j + j') - 2 \lfloor \frac{n+1}{2} \rfloor$,
\item if $j + j' - 2 \lfloor \frac{n+1}{2} \rfloor \ge 0$, then
\[
k\in [j+j' - 2 \lfloor \frac{n+1}{2} \rfloor +1, j' - \lfloor \frac{n+1}{2} \rfloor]\cup [2n-r+j,n],
\]
if $j+j' - 2  \lfloor \frac{n+1}{2} \rfloor < 0$, then
\[
k\in [1,j'- \lfloor \frac{n+1}{2} \rfloor]\cup [j-r+2n, j+j'-2r+3n-2].
\]
\end{itemize}
In particular, if both $(M_{J,J'})_{i,k}$ and $(M_{J,\widetilde{J'}})_{i,\widetilde{k}}$ are nonzero, then we must have
\begin{equation}\label{eqn: same row variables}
\widetilde{k} - k = \sum_{t=1}^{p+1} (\widetilde{j'_t} - j'_t) = (\widetilde{j}+\widetilde{j'})- (j+j'),
\end{equation}
where $J' = (j'_1<\cdots < j'_{p+1}) = J\setminus \{\lfloor (n+1)/2 \rfloor \}\cup \{j,j'\}$ and $\widetilde{J'} = (\widetilde{j'_1} <\dots < \widetilde{j'_{p+1}}) = J\setminus \{\lfloor (n+1)/2 \rfloor \}\cup \{\widetilde{j},\widetilde{j'}\}$. Similarly, if both $(M_{J,J'})_{i,k}$ and $(M_{\widetilde{J},J'})_{\widetilde{i},k}$ are nonzero, then we must have
\begin{equation}\label{eqn: same column variables}
\widetilde{i} - i = (\widetilde{j} + \widetilde{j'}) - (j+j'),
\end{equation}
where $J\setminus{\lfloor (n+1)/2 \rfloor} \cup \{j,j'\}= J' = \widetilde{J}\setminus{\lfloor (n+1)/2 \rfloor} \cup \{\widetilde{j},\widetilde{j'}\}$.

\emph{Step $3$:} We may write $J\in R$ as
\[
J = (j_1 < \cdots < j_{s-1} < j_{s} = \lfloor (n+1)/2 \rfloor< j_{s+1} < \cdots < j_{p} )
\]
and write $J'\in C$ as
\[
J' = (j'_1 < \cdots < j'_{t} < j'_{t+1} < \cdots < j'_{p+1}),
\]
where $j'_t < \lfloor (n+1)/2 \rfloor < j'_{t+1}$. By \eqref{eqn:blocks of M}, we see that in particular, $M_{J,J'} = 0$ if $s + 1\ne t $. This implies that the matrix $M$ is a block diagonal matrix $M = \operatorname{Diag}\{M_1,\dots, M_p\}$ where $M_s,1\le s \le \lfloor (n-1)/2 \rfloor $ is the submatrix of $M$ obtained by taking $J = (j_1 < \cdots < j_p)\in R$ and $J' = (j'_1 < \cdots < j'_{p+1})\in C$ such that
\[
j_s = \lfloor (n+1)/2 \rfloor, \quad j'_{s+1} < \lfloor (n+1)/2 \rfloor < j'_{s+2}.
\]
We define $R_s$ to be the subset of $R$ consisting of all $J = (j_1 < \cdots < j_p)\in R$ such that $j_s = \lfloor (n+1)/2 \rfloor$ and define $C_s$ to be the subset of $C$ consisting of all $J' = (j'_1 < \cdots < j'_{p+1})\in C$ such that $j'_{s} < \lfloor (n+1)/2 \rfloor < j'_{s+1}$.

For each $1\le s \le  \lfloor (n-1)/2 \rfloor $ and $J\in R_s,J'\in C_s$, we remove $(J,i)$-th row from $M_s$ where $p - s + 2 \le i \le n-s $ and we remove $(J',k)$-th column from $M_s$ where $p-s +2  \le k \le n-s $. We still denote the new matrix by $M_s$. We denote by $v_{J',k}$ the $(J',k)$-th column vector of $M_s$ if $J'\in C_s$. We use the same notation to denote column vectors of $M$ before, but since $M$ is a block diagonal matrix with diagonal blocks $M_0,\dots, M_p$, this abuse of the notation should cause no confusion.
We remark that the matrix $M_s$ is of size
\[
\binom{n-p-1}{s-1}\binom{p}{s}(p+1) \times \binom{n-p-1}{s} \binom{p}{s-1}(p+1)
\]
and thus it has more rows than columns. Hence it suffices to prove that the matrix $M_s$ has full rank, or equivalently, the set
\begin{equation}\label{eqn: linearly independent set}
S_s = \{ v_{J',k}: J'\in C_s, k =1,\dots, p-s+1,n-s+1,\dots,n\}
\end{equation}
is a linearly independent set.

\emph{Step $4$}: Let $s$ be an integer such that $1 \le s \le \lfloor (n-1)/2 \rfloor $ and let $S_s$ be the set defined in \eqref{eqn: linearly independent set}. To prove that $S_s$ is a linearly independent set, we consider
\begin{equation}\label{eqn:linear independence of column vectors}
\sum_{\substack{J'\in C_s\\   k =1,\dots, p-s+1,n-s+1,\dots, n}} x_{J',k} v_{J',k} = 0,
\end{equation}
where $x_{J',k}$'s are unknowns and we want to prove that $x_{J',k} = 0$ for all $J'\in C_s$ and $ k =1,\dots, p-s+1,n-s+1,\dots, n$. Since $v_{J',k}$ is the $(J',k)$-th column vector of $M_s$, the $(J,i)$-th entry $v_{J',k}^{J,i}$ of $v_{J',k}$ is equal to the $(i,k)$-th entry $(M_{J,J'})_{i,k}$ of $M_{J,J'}$ defined by \eqref{eqn: entries-1} and \eqref{eqn: entries-2}. Hence from \eqref{eqn:linear independence of column vectors} we have for each $J\in R_s$ and $1\le i \le n$ the following linear equation:
\begin{equation}\label{eqn:linear independence of column vectors-1}
\sum_{\substack{J'\in C_s\\  k =1,\dots, p-s+1,n-s+1,\dots, n}} (M_{J,J'})_{i,k} x_{J',k} = 0.
\end{equation}
We notice that if $(M_{J,J'})_{i,k} \ne 0$ and $k\le p-s+1$ (resp. $k\ge n-s+1$), then whenever $(M_{J,\widetilde{J'}})_{i,\widetilde{k}} \ne 0$, we must also have $\widetilde{k}\le p-s+1$ (resp. $\widetilde{k}\ge n-s+1$). This can be seen from \eqref{eqn: same row variables} and \eqref{eqn: vanishing of middle columns}. Hence \eqref{eqn:linear independence of column vectors-1} can be simplified as:
\begin{align}
\label{eqn:left equations}
\sum_{\substack{J'\in C_s\\  k =1,\dots, p-s+1}} (M_{J,J'})_{i,k} x_{J',k} &= 0 \quad  \text{ or }\\
\label{eqn:right equations}
\sum_{\substack{J'\in C_s\\  k =n-s+1,\dots, n}} (M_{J,J'})_{i,k} x_{J',k} &= 0.
\end{align}
We denote by $X^i_J$ the set of variables $x_{J',k}$ whose coefficient $(M_{J,J'})_{i,k}\ne 0$. Let $A=\{(J_1,i_1),\dots, (J_m,i_m)\}$ be a maximal set such that for any $(J_a,i_a)\in A$, there exists some $(J_b,i_b)\in A, a\ne b$,
\[
X^{i_a}_{J_a} \cap X^{i_b}_{J_b} \ne \emptyset.
\]
If $x_{J_0',k_0}\not\in \cup_{(J,i)\in A} X^i_J$ then we see that $x_{J_0',k_0}$ is independent on $x_{J',k} \in \cup_{(J,i)\in A} X^i_J$. Therefore, it is sufficient to prove that the solution of the linear system:
\begin{equation}\label{eqn: linearly independence of columns-2}
\sum_{x_{J',k}\in X^{i_t}_{J_t}} (M_{J_t,J'})_{i_t,k} x_{J',k} = 0,\quad t=1,\dots,m.
\end{equation}
is zero. By \eqref{eqn:left equations} and \eqref{eqn:right equations}, we see that
\begin{itemize}
\item either $k\le p-s+1$ for all $x_{J',k}\in \cup_{(J,i)\in A} X^i_J$,
\item or $k\ge n-s +1$ for all $x_{J',k}\in \cup_{(J,i)\in A} X^i_J$.
\end{itemize}
Similarly, we also have
\begin{itemize}
\item either $i\le \lfloor p -s + 1$ for all $(J,i)\in A$,
\item or $i \ge n-s+1$ for all $(J,i)\in A$.
\end{itemize}
If there exists $x_{J',k}\in \cup_{(J,i)\in A} X^i_J$ such that $k\ge n-s+1$ and $i\le p - s + 1$, then $k - i \ge n - p$. However, we have
\[
k - i = (j+j') - 2 \lfloor \frac{n+1}{2} \rfloor,
\]
this implies that
\[
j + j' \ge n + 2\lfloor \frac{n+1}{2} \rfloor - p \ge n + \lfloor \frac{n+1}{2} \rfloor,
\]
which contradicts the assumption that $j < \lfloor (n+1)/2 \rfloor < j'\le n $. Similarly, we may prove that if $k \le p-s+1$ and $i\ge n-s+1$, then
\[
j+j' \le p - n + 2\lfloor \frac{n+1}{2} \rfloor \le p + 1 \le \lfloor \frac{n+1}{2} \rfloor + 1,
\]
which contradicts the assumption that $1\le j <  \lfloor (n+1)/2 \rfloor < j'$. Hence if $(J,i)\in A$ and $i\le p - s + 1$ (resp. $i\ge n-s+1$), then all $x_{J',k}\in \cup_{(J,i)\in A} X^i_J$ must have $k\le p-s + 1$ (resp. $k\ge n - s + 1$). We denote by $E_1$ the set of integers $1,2,\dots, p - s + 1$ and by $E_2$ the set of integers $n-s+1,n-s+2,\dots, n$.

According to \eqref{eqn: same column variables} and \eqref{eqn: same row variables}, we may describe the set $A$ as follows: if $(J,i)\in A,i\in E_s,s=1,2$,  then $(\widetilde{J},\widetilde{i})\in A$ if and only if
\[
\widetilde{i} - i = \sum_{t=1}^p (\widetilde{j_t} - j_t ),\quad  \widetilde{i}\in E_s.
\]
The set $\cup_{(J,i)\in A} X^i_J$ can be described as in a similar way: if $x_{J',k}\in \cup_{(J,i)\in A} X^i_J, i,k\in E_s, s=1,2$, then $x_{\widetilde{J'},\widetilde{k}}\in \cup_{(J,i)\in A} X^i_J$ if and only if
\[
\widetilde{k} - k = \sum_{t=1}^{p+1} (\widetilde{j'_t} - j'_t),\quad \widetilde{k}\in E_s.
\]

\emph{Step $5$}: $A = \{(J_1,i_1),\dots, (J_m,i_m)\}$ be a maximal set such that for any $(J_a,i_a)\in A$ there exists some $(J_b,i_b)\in A$ such that
\[
X^{i_a}_{J_a} \cap X^{i_b}_{J_b} \ne \emptyset, \quad a\ne b.
\]
We prove that every $x_{J',k}\in \cup_{(J,i)\in A} X^i_J$ is equal to zero. To see this, it is sufficient to prove that the solution to the linear system \eqref{eqn: linearly independence of columns-2} over the field $\mathbb{Z}_2 =\{0,1\}$ must be trivial, i.e., $x_{J',k} = 0$. Indeed, if \eqref {eqn: linearly independence of columns-2} has a nontrivial solution over $\mathbb{C}$, then it also has a nontrivial solution over $\mathbb{Z}$ since coefficients of \eqref{eqn: linearly independence of columns-2} are $-1,1$ or $0$ and hence in particular are integers. Moreover, among these nontrivial integer solutions, there must be a solution $(a_{J',k})_{x_{J',k}\in \cup_{(J,i)\in A} X^i_J}$ such that
\[
a_{J',k} \equiv 1 \quad  (\mod 2)
\]
for some $(J',k)$. Equivalently, \eqref{eqn: linearly independence of columns-2} has a nontrivial solution over $\mathbb{Z}_2$.

\emph{Step $6$:} We denote by $M_{s,A}$ the coefficient matrix of the system \eqref{eqn: linearly independence of columns-2} and we suppose that $M_{s,A}$ is an $m\times l$ matrix. By the construction of $M_{s,A}$ we know that $m\ge l$. We define an order on column indices $\{(J'_1,k_1),\dots, (J'_l,k_l)\}$ of $M_{s,A}$ by $(J'_a,k_a) > (J'_b,k_b)$ if
\[
 \sum_{t=s+1}^{p+1} j'_{at} - \sum_{t=1}^s j'_{at} > \sum_{t=s+1}^{p+1} j'_{bt} - \sum_{t=1}^s j'_{bt}
 \]
or
\[
 \sum_{t=s+1}^{p+1} j'_{at} - \sum_{t=1}^s j'_{at} = \sum_{t=s+1}^{p+1} j'_{bt} - \sum_{t=1}^s j'_{bt}
 \text{ and } k_a > k_b,
\]
where
\begin{align*}
J'_a &= (j'_{a,1} < \cdots j'_{a,s} <j'_{a,s+1} < \cdots < j'_{a,p+1})\in C_s, \\
J'_b &= (j'_{b,1} < \cdots j'_{b,s} <j'_{b,s+1} < \cdots < j'_{b,p+1})\in C_s.
\end{align*}
Similarly, we may  define an order on the set $A = \{(J_1,i_1),\dots, (J_m,i_m)\}$ by $(J_a,i_a) > (J_b,i_b)$ if
\[
\sum_{t=s+1}^p j_{at} - \sum_{t=1}^{s-1} j_{at} > \sum_{t=s+1}^p j_{bt} - \sum_{t=1}^{s-1} j_{bt}
\]
or
\[
\sum_{t=s+1}^p j_{at} - \sum_{t=1}^{s-1} j_{at} = \sum_{t=s+1}^p j_{bt} - \sum_{t=1}^{s-1} j_{bt}
\text{ and } i_a > i_b,
\]
where
\begin{align*}
J_a &= (j_{a,1} < \cdots < j_{a,s-1} < j_{a,s} = \lfloor (n+1)/2 \rfloor < j_{a,s+1} < \cdots < j_{a,p})\in R_s, \\
J_b &= (j_{b,1} < \cdots < j_{b,s-1} < j_{b,s} = \lfloor (n+1)/2 \rfloor < j_{b,s+1} < \cdots < j_{b,p})\in R_s.
\end{align*}
With the order defined above, we may reorder $(J_1,i_1),\dots, (J_m,i_m)$ and $(J'_1,k_1),\dots, (J'_l,k_l)$ respectively so that
\begin{align*}
(J_a,i_a) &\le (J_b,i_b),\quad 1\le a < b \le m, \\
(J'_\alpha,k_\alpha) &\le (J'_\beta,k_\beta),\quad 1\le \alpha < \beta \le l.
\end{align*}

To prove that $\eqref{eqn: linearly independence of columns-2}$ only has a trivial solution over $\mathbb{Z}_2$, it suffices to prove that
\begin{itemize}
\item there exists $(J,i)\in A$ such that the $(J,i)$-th row vector $w$ of $M_{s,A}$  has exactly one nonzero entry indexed by $(J',k)$ and
\item The submatrix obtained by removing $(J,i)$-th row and $(J',k)$-th column from $M_{s,A}$ has full rank.
\end{itemize}
In fact, the row vector $w$ is one of the following:
\begin{itemize}
\item $w$ is the $(J_1,i_1)$-th row of $M_{s,A}$.
\item $w$ is the $(J_m,i_m)$-th row of $M_{s,A}$ .
\item $w$ is the $(J,i)$-th row of $M_{s,A}$, where $(J,i)$ is maximal among those such that $(M_{s,A})_{(J,i),(J'_1,k_1)}$.
\item $w$ is the $(J,i)$-th row of $M_{s,A}$, where $(J,i)$ is minimal among those such that $(M_{s,A})_{(J,i),(J'_1,k_1)}$.
\item $w$ is the $(J,i)$-th row of $M_{s,A}$, where $(J,i)$ is maximal among those such that $(M_{s,A})_{(J,i),(J'_l,k_l)}$.
\item $w$ is the $(J,i)$-th row of $M_{s,A}$, where $(J,i)$ is minimal among those such that $(M_{s,A})_{(J,i),(J'_l,k_l)}$.
\end{itemize}
We denote by $M^1_{s,A}$ the matrix obtained by removing the $(J,i)$-th row and $(J',k)$-th column from $M_{s,A}$. Then $M^1_{s,A}$ has a row vector which contains exactly one nonzero entry. Indeed, we can find this row vector in the same way as we find the row vector $w$ for $M_{s,A}$. Hence by induction, we see that $M_{s,A}$ must have full rank and this completes the proof.
\end{proof}

We illustrate the proof of Lemma \ref{lemma: technical lemma} by some examples.
\begin{example}\label{example: s=1}
By the definition of $M_{s,A}$, we see that for any positive integer $n$, if $s=1$ or $\lfloor (n-1)/2 \rfloor$, then $M_{s,A}$ is simply a $p\times p$ triangular matrix whose diagonal entries are all $1$. In this case, we see that $M_{s,A}$ obviously has full rank.
\end{example}

\begin{example}\label{example: n=6,s=2}
Let $n= 6$ and $s=2$, then $p =\left\lfloor \frac{n}{2} \right\rfloor = 3$,
\[
\left\lfloor \frac{n+1}{2} \right\rfloor = 3,\quad p+1-s = 2,\quad n-s+1 =5.
\]
We also have
\begin{align*}
C_s &= \{(1,2,4,5), (1,2,4,6), (1,2,5,6)\}, \\
R_s &= \{(1,3,4),(1,3,5),(1,3,6),(2,3,4),(2,3,5),(2,3,6)\}.
\end{align*}
The matrix $M_s$ is
\[
\arraycolsep=2.5pt
M_s =  \kbordermatrix{
       & (1,2,4,5) & (1,2,4,6) & (1,2,5,6) \\
(1,3,4) & T_{-1,2} & T_{-1,3} & 0 \\
(1,3,5) & T_{-1,1} & 0 & T_{-1,3} \\
(1,3,6) & 0 & T_{-1,1} & T_{-1,2}\\
(2,3,4) & T_{-2,2} & T_{-2,3} & 0 \\
(2,3,5) & T_{-2,1} & 0 & T_{-2,3}\\
(2,3,6) & 0 & T_{-2,1} & T_{-2,2}
},
\]
where $T_{a,b}$ is the submatrix obtained by removing columns $k=3,4$ and rows $i = 3,4$ from $T_aT_b - T_bT_a$ over $\mathbb{Z}_2$. Here $T_a$ is the $n\times n$ Toeplitz matrix whose $(u,v)$-th entry is one if $v-u = a$ and is zero otherwise. For example, $T_{-1,2}$  and $T_{-2,2}$ are
\[
T_{-1,2} =
\begin{bmatrix}
0 & 1   & 0 & 0 \\
0 & 0   & 0 & 1 \\
0 & 0   & 0 & 0 \\
0 & 0   & 0 & 0 \\
\end{bmatrix}, \quad T_{-2,2} =
\begin{bmatrix}
1 & 0  & 0 & 0 \\
0 & 1 &  0 & 0 \\
0 & 0  & 1 & 0 \\
0 & 0  & 0 & 1 \\
\end{bmatrix}.
\]
If we take
\[
A = \{((2,3,4),1), ((2,3,5),2), ((1,3,5),1), ((1,3,6),2) \},
\]
then the matrix $M_{s,A}$ is
\[
\arraycolsep=2.5pt
M_{s,A} =  \kbordermatrix{
       & ((1,2,4,5),1) & ((1,2,4,6),2)  \\
((2,3,4),1) & 1 & 1 \\
((2,3,5),2) & 1 & 0 \\
((1,3,5),1) & 1 & 0  \\
((1,3,6),2) & 0 & 0
}.
\]
It is straightforward to verify that the $((1,3,5),1)$-th row of $M_{s,A}$ has only one nonzero entry, which is indexed by $((1,2,4,5),1)$. We delete the $((1,3,5),1)$-th row and the $((1,2,4,5),1)$-th column of $M_{s,A}$ to obtain $M^1_{s,A} = \begin{bmatrix}
1
\end{bmatrix}$, which has full rank.
\end{example}


\begin{example}\label{example: n=7,s=2}
We let $n=7$, $s=2$ then $p = 3$,
\[
\left\lfloor \frac{n+1}{2} \right\rfloor =4,\quad p+1-s = 2,\quad n-s+1 = 6.
\]
We take
\[
A = \{((2,4,5),1),((3,4,5),2),((1,4,6),1),((2,4,6),2),((1,4,7),2)\}.
\]
The matrix $M_{s,A}$ is
\[
\arraycolsep=2.5pt
M_{s,A} =  \kbordermatrix{
       & ((2,3,5,6),2) & ((1,3,5,6),1)  & ((1,3,5,7),2)  & ((1,2,5,7),1) & ((1,2,6,7),2)\\
((3,4,5),2) & 1& 1  & 1  &0 & 0 \\
((2,4,5),1)&1 & 0 & 0  & 1 & 0 \\
((2,4,6),2) & 0& 0  & 0  & 0 & 1\\
((1,4,6),1) & 0& 1  & 0  & 0 & 1\\
((1,4,7),2) & 0& 0  & 0   & 1 & 1\\
}.
\]
We see that the $((2,4,6),2)$-th row has the unique nonzero entry indexed by $((1,2,6,7),2)$. We obtain $M^1_{s,A}$ by removing $((2,4,6),2)$-th row and $((1,2,6,7),2)$-th column:
\[
\arraycolsep=2.5pt
M^1_{s,A} =  \kbordermatrix{
       & ((2,3,5,6),2) & ((1,3,5,6),1)  & ((1,3,5,7),2) ) & ((1,2,5,7),1) \\
((3,4,5),2) & 1& 1  & 1  &0 \\
((2,4,5),1)&1 & 0 & 0 & 1  \\
((1,4,6),1) & 0& 1  & 0  & 0\\
((1,4,7),2) & 0& 0  & 0   & 1\\
}.
\]
By the same procedure, we obtain
\[
\arraycolsep=2pt
M^2_{s,A} =  \kbordermatrix{
       & ((2,3,5,6),2) & ((1,3,5,6),1)  & ((1,3,5,7),2) )  \\
((3,4,5),2) & 1& 1  & 1   \\
((2,4,5),1)&1 & 0 & 0 \\
((1,4,6),1) & 0& 1  & 0  \\
}
\]
by removing $((1,4,7),1)$-th row and $((1,2,5,7),1)$-th column of $M^1_{s,A}$ and we obtain
\[
\arraycolsep=2pt
M^3_{s,A} =  \kbordermatrix{
       & ((1,3,5,6),1)  & ((1,3,5,7),2) )  \\
((3,4,5),2) & 1  & 1   \\
((1,4,6),1) & 1  & 0  \\
}, \quad
\arraycolsep=1pt
M^4_{s,A} =  \kbordermatrix{
       & ((1,3,5,7),2) )  \\
((3,4,5),2)   & 1   \\
},
\]
by removing $((2,4,5),1)$-th row and $((2,3,5,6),2)$-th column of $M^2_{s,A}$ and removing $((1,4,6),1)$-th row and $((1,3,5,6),1)$-th column of $M^3_{s,A}$, respectively.
\end{example}

\begin{example}\label{example: n=8,s=2}
Let $n=8, s =2$ then $p = \left\lfloor n/2 \right\rfloor =4$ and
\[
\left\lfloor \frac{n+1}{2} \right\rfloor = 4, \quad p+1 -s = 3, \quad  n-s +1 = 7.
\]
We consider
\begin{align*}
A = \{&((3,4,5,6),1), ((3,4,5,7),2), ((3,4,6,7),3), ((2,4,5,7),1),((3,4,5,8),3),\\&((2,4,5,8),2), ((1,4,6,7),1),((2,4,6,8),3),((1,4,6,8),2),((1,4,7,8),3)\}.
\end{align*}
The corresponding $M_{s,A}$ is
\[
\tiny
\arraycolsep=0.5pt
M_{s,A} =  \kbordermatrix{
       & ((2,3,5,6,7),2) & ((1,3,5,6,7),1)  & ((2,3,5,6,8),3)  & ((1,3,5,6,8),2) & ((1,2,5,6,8),1) & ((1,3,5,7,8),3)  & ((1,2,5,7,8),2) &((1,2,6,7,8),3) \\
((3,4,5,6),1) & 1  & 1 & 1 & 1 & 0 & 0  & 0 & 0   \\
((3,4,5,7),2) & 1  & 1 & 0 & 0 & 0 & 1  & 0 & 0  \\
((3,4,6,7),3) & 0  & 1 & 0 & 0 & 0 & 0  & 0 & 0  \\
((2,4,5,7),1) & 1  & 0 & 0 & 0 & 0 & 0  & 1 & 0 \\
((3,4,5,8),3) & 0  & 0 & 0 & 1 & 0 & 1  & 0 & 0  \\
((2,4,5,8),2) & 0  & 0 & 0 & 0 & 1 & 0  & 1 & 0   \\
((1,4,6,7),1) & 0  & 1 & 0 & 0 & 0 & 0  & 0 & 1   \\
((2,4,6,8),3) & 0  & 0 & 0 & 0 & 1 & 0  & 0 & 1  \\
((1,4,6,8),2) & 0  & 0 & 0 & 0 & 1 & 0  & 0 & 1 \\
((1,4,7,8),3) & 0  & 0 & 0 & 0 & 0 & 0  & 1 & 0
}
\]
and we remove $((1,4,7,8),3)$-th row and $((1,2,5,7,8),2)$-th column of $M_{s,A}$ to obtain
\[
\tiny
\arraycolsep=0.5pt
M^1_{s,A} =  \kbordermatrix{
       & ((2,3,5,6,7),2) & ((1,3,5,6,7),1)  & ((2,3,5,6,8),3)  & ((1,3,5,6,8),2) & ((1,2,5,6,8),1) & ((1,3,5,7,8),3)  &((1,2,6,7,8),3) \\
((3,4,5,6),1) & 1  & 1 & 1 & 1 & 0 & 0   & 0   \\
((3,4,5,7),2) & 1  & 1 & 0 & 0 & 0 & 1   & 0  \\
((3,4,6,7),3) & 0  & 1 & 0 & 0 & 0 & 0   & 0  \\
((2,4,5,7),1) & 1  & 0 & 0 & 0 & 0 & 0   & 0 \\
((3,4,5,8),3) & 0  & 0 & 0 & 1 & 0 & 1   & 0  \\
((2,4,5,8),2) & 0  & 0 & 0 & 0 & 1 & 0   & 0   \\
((1,4,6,7),1) & 0  & 1 & 0 & 0 & 0 & 0   & 1   \\
((2,4,6,8),3) & 0  & 0 & 0 & 0 & 1 & 0   & 1  \\
((1,4,6,8),2) & 0  & 0 & 0 & 0 & 1 & 0   & 1
}.
\]
We remove $((2,4,5,7),1)$-th row and $((2,3,5,6,7),2)$-th column of $M^1_{s,A}$ to obtain
\[
\tiny
\arraycolsep=0.5pt
M^2_{s,A} =  \kbordermatrix{
        & ((1,3,5,6,7),1)  & ((2,3,5,6,8),3)  & ((1,3,5,6,8),2) & ((1,2,5,6,8),1) & ((1,3,5,7,8),3)  &((1,2,6,7,8),3) \\
((3,4,5,6),1)   & 1 & 1 & 1 & 0 & 0   & 0   \\
((3,4,5,7),2)   & 1 & 0 & 0 & 0 & 1   & 0  \\
((3,4,6,7),3)   & 1 & 0 & 0 & 0 & 0   & 0  \\
((3,4,5,8),3)   & 0 & 0 & 1 & 0 & 1   & 0  \\
((2,4,5,8),2)   & 0 & 0 & 0 & 1 & 0   & 0   \\
((1,4,6,7),1)   & 1 & 0 & 0 & 0 & 0   & 1   \\
((2,4,6,8),3)   & 0 & 0 & 0 & 1 & 0   & 1  \\
((1,4,6,8),2)   & 0 & 0 & 0 & 1 & 0   & 1
}.
\]
We remove $((3,4,6,7),3)$-th row and $((1,3,5,6,7),1)$-th column of $M^3_{s,A}$ to obtain
\[
\tiny
\arraycolsep=0.5pt
M^3_{s,A} =  \kbordermatrix{
         & ((2,3,5,6,8),3)  & ((1,3,5,6,8),2) & ((1,2,5,6,8),1) & ((1,3,5,7,8),3)  &((1,2,6,7,8),3) \\
((3,4,5,6),1)    & 1 & 1 & 0 & 0   & 0   \\
((3,4,5,7),2)    & 0 & 0 & 0 & 1   & 0  \\
((3,4,5,8),3)    & 0 & 1 & 0 & 1   & 0  \\
((2,4,5,8),2)    & 0 & 0 & 1 & 0   & 0   \\
((1,4,6,7),1)    & 0 & 0 & 0 & 0   & 1   \\
((2,4,6,8),3)    & 0 & 0 & 1 & 0   & 1  \\
((1,4,6,8),2)    & 0 & 0 & 1 & 0   & 1
}.
\]
We remove $((1,4,6,7),1)$-th row and $((1,2,6,7,8),3)$-th column of $M^3_{s,A}$ to obtain
\[
\tiny
\arraycolsep=0.5pt
M^4_{s,A} =  \kbordermatrix{
         & ((2,3,5,6,8),3)  & ((1,3,5,6,8),2) & ((1,2,5,6,8),1) & ((1,3,5,7,8),3)  \\
((3,4,5,6),1)    & 1 & 1 & 0 & 0      \\
((3,4,5,7),2)    & 0 & 0 & 0 & 1     \\
((3,4,5,8),3)    & 0 & 1 & 0 & 1     \\
((2,4,5,8),2)    & 0 & 0 & 1 & 0      \\
((2,4,6,8),3)    & 0 & 0 & 1 & 0     \\
((1,4,6,8),2)    & 0 & 0 & 1 & 0
}.
\]
It is easy to determine $M^5_{s,A},M^6_{s,A}$ and lastly $M^7_{s,A} = \begin{bmatrix}
1 & 0 & 0
\end{bmatrix}^\mathsf{T}$.
\end{example}

\end{document}